\documentclass[12pt,amssymb,verbatim]{amsart}

\usepackage{amsfonts,latexsym,geometry,color,hyperref}
\usepackage{amssymb,graphics,graphicx}
\usepackage{amsfonts,latexsym,geometry}
\usepackage{amssymb,graphics,graphicx}
\usepackage{amscd}
\usepackage{amssymb}
\usepackage{amsthm}
\usepackage[utf8]{inputenc} 
\usepackage[T1]{fontenc}
\usepackage{amsmath}
\usepackage{amsfonts}
\usepackage{amssymb}
\usepackage{graphicx} 
\usepackage{enumerate}
\newcommand\supp{\mathrm{supp}}
\newtheorem{theoreme}{Theorem}[section] %
\newtheorem{proposition}[theoreme]{Proposition} %
\newtheorem{corollary}[theoreme]{Corollary} %
\newtheorem{lemme}[theoreme]{Lemma} %
\newtheorem{definition}{Definition}[section] %
\newtheorem{remark}[theoreme]{Remark} %
\newtheorem{remarque}[theoreme]{Remark} %
\newtheorem{example}[theoreme]{Example} %

\newcommand\normi{\shortparallel}
\newcommand\mk{\medskip}
\newcommand\sk{\smallskip}
\newcommand\N {\mathbb{N}}
\newcommand\R{\mathbb{R}} 
 
\newcommand\dimm{\underline{\dim}_H}

\RequirePackage{yhmath} 
\renewcommand\widering[1]{\ring{#1}}

 \author{Edouard Daviaud}

\begin{document}

\title{Shrinking target for $C^{1}$ weakly conformal IFS with overlaps}
\maketitle
Edouard DAVIAUD, Université Paris-Est, LAMA (UMR 8050) UPEMLV, UPEC, CNRS, F-94010, Créteil, France.

\begin{abstract}
In this article, we study the Hausdorff dimension of weakly conformal IFS's shrinking targets with possible overlaps, provided the conformality dimension of the systems and the dimension of the attractor are equal. Those results extends the works of Hill-Velani as well as the results obtained in \cite{ED3} for self-similar IFS's.
\end{abstract}
\section{Introduction}
Estimating the Hausdorff dimension of points falling infinitely many often in sets $U_n$ having some algebraic or dynamical meaning is a question which arises naturally in Diophantine approximation as well as in dynamical systems. Given a metric space $(X,d)$, a measurable mapping $T:X \to X$ and an ergodic probability measure $\mu$, the usual question consists in estimating, for $\mu$-typical points $x$, the Hausdorff dimension of points falling infinitely many often in balls $B(T^n (x),\phi(n))$, centered in $T^n (x)$ and with radius $\phi(n).$  Such problems have been studied for instance in \cite{LS,FanSchmTrou, HV,LS,Baker,AllenB,PR,ED3} and are called ``shrinking targets problems''. 

Estimating these dimensions often relies on establishing ubiquity theorems (or mass transference principles) for the ergodic probability measure $\mu$. Given a sequence of balls $(B_n :=B(x_n, r_n))_{n\in\mathbb{N}}$, these theorems usually aims at giving lower-bounds for the dimension of sets of points of the form $\limsup_{n\rightarrow+\infty}U_n$, where $U_n \subset B_n$ (typically, $U_n =B_n  ^{\delta}=B(x_n ,r_n ^{\delta})$), provided that the sequence of balls $(B_n)_{n\in\mathbb{N}}$ satisfies $\mu\Big(\limsup_{n\rightarrow+\infty}B_n\Big)=1.$ 

Let $m\geq 2$ be an integer and $S=\left\{f_1 ,...,f_m\right\}$ be a weakly conformal system of $m$  $C^{1}$ contracting maps from $\mathbb{R}^d \to \mathbb{R}^d$ (Definition \ref{confdef}) . Denote by $K$ the attractor of $S$, i.e the unique non empty compact set satisfying $K=\bigcup_{i=1}^m f_i (K),$ $\Lambda=\left\{1,...,m\right\}$, $\Lambda^* =\bigcup_{k\geq 0}\Lambda^k$ and, for $k\in\mathbb{N}$, $\underline{i}=(i_1 ,...,i_k)\in\Lambda^k$, write $f_{\underline{i}}=f_{i_1}\circ ... \circ f_{i_k}.$ 

In this article, we prove that if the dimension of $K$ can be computed in the same way than if the open set condition holds for $S$, (meaning that   $\dim_H (K)=\dim(S)$,  Definition  \ref{expdim}), then, for any $x\in K$, for any $\delta \geq 1$, one has 
\begin{equation}
\label{stag}
\dim_H \Big(\limsup_{\underline{i}\in\Lambda^{*}}B(f_{\underline{i}}(x),\vert f_{\underline{i}}(K)\vert^ { \delta})\Big)=\frac{\dim_H (K)}{\delta}.
\end{equation} 

In other words, the dimension of points $y$ for which the orbit of $x$, $\Big(f_{\underline{i}}(x) \Big)_{\underline{i} \in\Lambda^*},$ verifies infinitely many often $d(y,f_{\underline{i}}(x))\leq \vert f_{\underline{i}}(K)\vert^ { \delta}$ has  dimension $ \frac{\dim_H (K)}{\delta}.$ 

Similar results are established in \cite{AllenB,HV} under the open set condition and some particular cases of overlapping self-similar systems  are treated in \cite{Baker,ED3} and, as an application of our approach, a complement of some results established in \cite{Baker} are also given (see Theorem \ref{compbaker}). One emphasizes that the condition  $\dim_H (K)=\dim(S)$ is much weaker than the open set condition. For instance this condition is satisfied for self-similar systems in $\mathbb{R}$, as soon as Hochman's exponential separation condition is verified, so in particular if the contraction ratios and the translation parameters are algebraic numbers. 

The key tool to establish \eqref{stag} is Theorem \ref{prop-ss}, which is a ubiquity theorem for weakly conformal measures (without any separation condition) and strongly relies on the technics developed in \cite{ED3}.
\section{Definitions and main statements} 

Let us start with some notations 

 Let $d$ $\in\mathbb{N}$. For $x\in\mathbb{R}^{d}$, $r>0$,  $B(x,r)$ stands for the closed ball of ($\mathbb{R}^{d}$,$\vert\vert \ \ \vert\vert_{\infty}$) of center $x$ and radius $r$. 
 Given a ball $B$, $\vert B\vert$ stands for the diameter of $B$. For $t\geq 0$, $\delta\in\mathbb{R}$ and $B=B(x,r)$,   $t B$ stands for $B(x,t r)$, i.e. the ball with same center as $B$ and radius multiplied by $t$,   and the  $\delta$-contracted  ball $B^{\delta}$ is  defined by $B^{\delta}=B(x ,r^{\delta})$.
\smallskip

Given a set $E\subset \mathbb{R}^d$, $\widering{E}$ stands for the  interior of the $E$, $\overline{E}$ its  closure and $\partial E =\overline{E}\setminus \widering{E}$ its boundary. If $E$ is a Borel subset of $\R^d$, its Borel $\sigma$-algebra is denoted by $\mathcal B(E)$.
\smallskip

Given a topological space $X$, the Borel $\sigma$-algebra of $X$ is denoted $\mathcal{B}(X)$ and the space of probability measure on $\mathcal{B}(X)$ is denoted $\mathcal{M}(X).$ 

\sk
 The $d$-dimensional Lebesgue measure on $(\mathbb R^d,\mathcal{B}(\mathbb{R}^d))$ is denoted by 
$\mathcal{L}^d$.
\smallskip

For $\mu \in\mathcal{M}(\R^d)$,   $\supp(\mu)=\left\{x\in[0,1]: \ \forall r>0, \ \mu(B(x,r))>0\right\}$ is the topological support of $\mu$.
\smallskip

 Given $E\subset \mathbb{R}^d$, $\dim_{H}(E)$ and $\dim_{P}(E)$ denote respectively  the Hausdorff   and the packing dimension of $E$.
\smallskip

Let $f:\mathbb{R}^d \to \mathbb{R}^d$ be a linear map. One sets
$\vert \vert f\vert \vert=\sup_{x\in B(0,1)}\frac{\vert\vert f(x)\vert\vert_{\infty}}{\vert  \vert x\vert \vert_{\infty}}$ and $\normi f\normi=\inf_{x\in B(0,1)}\frac{\vert\vert f(x)\vert\vert_{\infty}}{\vert  \vert x\vert \vert_{\infty}}.$
\medskip

Now we recall some definitions.


\begin{definition}
\label{hausgau}
Let $\zeta :\mathbb{R}^{+}\mapsto\mathbb{R}^+$. Suppose that $\zeta$ is increasing in a neighborhood of $0$ and $\zeta (0)=0$. The  Hausdorff outer measure at scale $t\in(0,+\infty]$ associated with the gauge $\zeta$ of a set $E$ is defined by 
\begin{equation}
\label{gaug}
\mathcal{H}^{\zeta}_t (E)=\inf \left\{\sum_{n\in\mathbb{N}}\zeta (\vert B_n\vert) : \, \vert B_n \vert \leq t, \ B_n \text{ closed ball and } E\subset \bigcup_{n\in \mathbb{N}}B_n\right\}.
\end{equation}
The Hausdorff measure associated with $\zeta$ of a set $E$ is defined by 
\begin{equation}
\mathcal{H}^{\zeta} (E)=\lim_{t\to 0^+}\mathcal{H}^{\zeta}_t (E).
\end{equation}
\end{definition}

For $t\in (0,+\infty]$, $s\geq 0$ and $\zeta:x\mapsto x^s$, one simply uses the usual notation $\mathcal{H}^{\zeta}_t (E)=\mathcal{H}^{s}_t (E)$ and $\mathcal{H}^{\zeta} (E)=\mathcal{H}^{s} (E)$, and these measures are called $s$-dimensional Hausdorff outer measure at scale $t\in(0,+\infty]$ and  $s$-dimensional Hausdorff measure respectively. Thus, 
\begin{equation}
\label{hcont}
\mathcal{H}^{s}_{t}(E)=\inf \left\{\sum_{n\in\mathbb{N}}\vert B_n\vert^s : \, \vert B_n \vert \leq t, \ B_n \text{  closed ball and } E\subset \bigcup_{n\in \mathbb{N}}B_n\right\}. 
\end{equation}
The quantity $\mathcal{H}^{s}_{\infty}(E)$ (obtained for $t=+\infty$) is called the $s$-dimensional Hausdorff content of the set $E$.

\begin{definition} 
\label{dim}
Let $\mu\in\mathcal{M}(\mathbb{R}^d)$.  
For $x\in \supp(\mu)$, the lower and upper  local dimensions of $\mu$ at $x$ are  defined as
\begin{align*}
\underline\dim_{{\rm loc}}(\mu,x)=\liminf_{r\rightarrow 0^{+}}\frac{\log(\mu(B(x,r)))}{\log(r)}
 \mbox{ and } \ \    \overline\dim_{{\rm loc}}(\mu,x)=\limsup_{r\rightarrow 0^{+}}\frac{\log (\mu(B(x,r)))}{\log(r)}.
 \end{align*}
Then, the lower and upper Hausdorff dimensions of $\mu$  are defined by 
\begin{equation}
\label{dimmu}
\dimm(\mu)={\mathrm{ess\,inf}}_{\mu}(\underline\dim_{{\rm loc}}(\mu,x))  \ \ \mbox{ and } \ \ \overline{\dim}_P (\mu)={\mathrm{ess\,sup}}_{\mu}(\overline\dim_{{\rm loc}}(\mu,x))
\end{equation}
respectively.
\end{definition}

It is known (for more details see \cite{F}) that
\begin{equation*}
\begin{split}
\dimm(\mu)&=\inf\{\dim_{H}(E):\, E\in\mathcal{B}(\mathbb{R}^d),\, \mu(E)>0\} \\
\overline{\dim}_P (\mu)&=\inf\{\dim_P(E):\, E\in\mathcal{B}(\mathbb{R}^d),\, \mu(E)=1\}.
\end{split}
\end{equation*}
When $\underline \dim_H(\mu)=\overline \dim_P(\mu)$, this common value is simply denoted by $\dim(\mu)$ and~$\mu$ is said to be \textit{exact dimensional}.

\medskip

\medskip

\mk

 Let us recall the notion of attractor and invariant measures for contracting IFS's.

\begin{definition}

%
%

Let $X$ be a compact subset of $\R^d$. A map $f$ is a $C^{1}$  contraction on $X$ if $f(X)\subset X$ and there is an open set $U$ such that $f$ is a $C^1$ diffeomormphism on $U$ and $\sup_{x\in X}\|f'(x)\|<1$.  
\end{definition}

\begin{definition}
\label{def-ssmu} Let $m \geq 2$ be an integer. A system $S=\left\{f_i\right\}_{i=1}^m$ of $m$ $C^{1}$ contraction from a compact set $X$ to $X$ is called an iterated function system (in short, IFS).

Let $(p_i)_{i=1,...,m}\in (0,1)^m$ be a positive probability vector, i.e. $p_1+\cdots +p_m=1$.

There exists a unique probability measure $\mu$ satisfying
\begin{equation}
\label{def-ssmu2}
\mu=\sum_{i=1}^m p_i \mu \circ f_i^{-1}.
\end{equation}

The topological support of $\mu$ is the attractor of $S$, that is the unique non-empty compact set $K\subset X$ such that  $K=\bigcup_{i=1}^m f_i(K)$. 

Denote by $\pi$ the canonical projection of $\left\{1,...,m\right\}^{\mathbb{N}}$, defined by
\begin{equation}
\pi((x_n)_{k\in\mathbb{N}})=\lim_{k\rightarrow+\infty}f_{x_1}\circ f_{x_2}\circ ...\circ f_{x_k}(0).
\end{equation}
\end{definition}
The existence  and uniqueness of $K$ and $\mu$ are standard results \cite{Hutchinson}.

\sk

We now recall the definition of a weakly conformal map, introduce by Feng in (see  \cite{FH} for example) which we will be particularly interested in this article.

\begin{definition}
\label{confdef}
%
%

 Let $m \geq 2$ be an integer, $S=\left\{f_i\right\}_{i=1}^m$ of $m$ $C^{1}$ contractions from an open  set $U$ to $U$ and $K$ its attractor.

%


One says that $S$ is weakly conformal if $S$ verifies 
\begin{equation}
\lim_{k\rightarrow+\infty}\sup_{(x_i)_{i\in\mathbb{N}}\in\left\{1,...,m\right\}^{\mathbb{N}}}\frac{1}{k}\Big(\log \normi f_{(x_1 ,...,x_k)}^{\prime}(\pi (\sigma^k (x)))\normi-\log \vert \vert f_{(x_1 ,...,x_k)}^{\prime}(\pi (\sigma^k (x)))\vert\vert \Big)=0.
\end{equation} 
In this case, a measure defined by \eqref{def-ssmu2} is called a weakly conformal measure.
\end{definition}


 Recall that due to a result by Feng and Hu \cite{FH} any weakly conformal measure is exact dimensional. 

\begin{example}
\begin{itemize}
\item[•] If the maps $f_1 ,...,f_m$ are affine similarities or conformal maps (i.e verify $\vert \vert f^{\prime}(x)(y)\vert \vert=\vert \vert f^{\prime}(x)\vert\vert \cdot\vert\vert y\vert\vert$ for every $x\in U,y\in\mathbb{R}^d$) the system  $S=\left\{f_1,...,f_m\right\}$ is weakly conformal. In this case  the IFS is called self-similar or self-conformal and the measures satisfying \eqref{def-ssmu2} are called respectively, self-similar and self-conformal measures. Note that this class of IFS contains for instance every system of holomorphic contracting mappings.\sk
\item[•] Assume that for any $1\leq i\leq m,$ $f_i: \mathbb{R}^d \to \mathbb{R}^d$ is defined by $f_i(x)=A_i x +b_i ,$ where for any $1\leq i\leq m$, $b_i \in\mathbb{R}^d$ and $A_i \in GL_d (\mathbb{R})$ has its eigenvalue equal in modulus to $0<r_i<1$ and for any $1 \leq i,j \leq m$, $A_i A_j=A_j A_i$. Then $S=\left\{f_1,...,f_m\right\}$ is weakly conformal. 
\end{itemize}
\end{example}

\sk


\sk

The pressure, as defined below, plays a particular role in the dimension theory of  weakly conformal IFS's. It is defined by the following proposition which will be proved in section \ref{sec-geo}.
 \begin{proposition}
 \label{propopres}
 Let  $m\geq 2$ be an integer,  $S=\left\{f_1 ,...,f_m\right\}$ be $C^{1}$ weakly conformal  IFS and $K$ its attractor.
 
 Let us fix $s\geq 0$ and $z\in K.$ The following quantity is well defined and independent of the choice of $z:$ 
 \begin{equation}
 \label{defpression}
P_z(s)=\lim_{k\rightarrow+\infty}\frac{1}{k}\log \sum_{\underline{i}\in \Lambda^k}\vert \vert f_{\underline{i}}^{\prime}(z)\vert \vert^s.
 \end{equation} 

 \end{proposition}

The conformality dimension of $S$ is defined as follows.
\begin{definition}
\label{expdim}
Let  $m\geq 2$ be an integer. Let $S=\left\{f_1 ,...,f_m\right\}$ be $C^{1}$ weakly conformal  IFS and $K$ its attractor.

Let us denote by $\dim(S)$ the unique solution to 
$$P(s)=0.$$ 
 One says that $\dim(S)$ is the conformality dimension of $S$.

\end{definition}
\begin{remarque}
If the mappings $f_1 ,...,f_m$ are similarities, the conformality dimension is called the similarity dimension. It is simply the unique real number solution of 
\begin{equation}
\label{simdim}
\sum_{i=1}^{m}c_i ^{\dim( S)}=1.
\end{equation}
\end{remarque}
Our main result is the following.
\begin{theoreme}
\label{SHRTARG}
Let $m\geq 2$ be an integer. Let $S=\left\{f_1 ,...,f_m\right\}$ be a $C^{1}$ weakly conformal  IFS of an open set $U$ with attractor $K$. Then, for any $x\in U,$ for any $\delta<1$, $$\limsup_{\underline{i} \in\Lambda^*}B(f_{\underline{i}}(x_0),\vert f_{\underline{i}}(K)\vert^{\delta})=K.$$
For any $x_0\notin K$, for any $\delta>1$,
\begin{equation}
\limsup_{\underline{i} \in\Lambda^*}B(f_{\underline{i}}(x_0),\vert f_{\underline{i}}(K)\vert^{\delta})=\emptyset.
\end{equation}
Assume $\dim_H (K)=\dim(S)$. Then, using the notation of Definition \ref{expdim}, for any $x_0\in K$, for any $\delta\geq 1$ it holds that
\begin{equation}
\dim_H \left(\limsup_{\underline{i}\in\Lambda^{*}}B(f_{\underline{i}}(x_0),\vert f_{\underline{i}}(K)\vert^{\delta})\right)=\frac{\dim_H (K)}{\delta}.
\end{equation}

\end{theoreme}

%
%
%
%
%
%
%
%
%
%
%

Some cases of self-similar shrinking targets with overlaps are studied in \cite{Baker}. The following theorem is proved.
\begin{theoreme}[\cite{Baker}]
\label{resbaker}
Let $m\geq 2$ and $S=\left\{f_1,...,f_m\right\}$ be a system of $m$ similarities of contraction ratio $0<c_1 ,...,c_m <1.$ Let $\mu\in\mathcal{M}(\mathbb{R}^d)$ be the self-similar measure solution to 
\begin{equation}
\label{defmusim}
\mu(\cdot)=\sum_{i=1}^{m}c_i ^{\dim(S)}\mu(f_{i}^{-1}(\cdot)).
\end{equation}
 Let $g:\mathbb{N}\to (0,+\infty),$ be a non increasing mapping.  Assume that 
 
\begin{equation}
 \label{condibaker}
 \begin{cases}
 \sum_{i=1}^{m}-c_i^{\dim(S)}\log(c_i  ^{\dim(S)})<-2\log\Big( \sum_{i=1}^m c_i ^{2\dim(S)}\Big)\text{ or }\\
 c_1=...=c_m.
 \end{cases} 
\end{equation}  
   If 
 $$\sum_{k\in\mathbb{N}}\sum_{\underline{i}\in\Lambda^k}k \Big(\vert f_{\underline{i}}(K)\vert g(k)\Big)^{\dim(S)}=+\infty, $$
  then $\mu\Big(\limsup_{\underline{i}\in\Lambda^{*}}B(f_{\underline{i}}(x),\vert f_{ \underline{i}}(K)\vert g(\vert \underline{i}\vert))\Big)=1.$
\end{theoreme}
 Theorem \ref{resbaker} can now be completed.
\begin{theoreme}
\label{compbaker}
Let $g:\mathbb{N}\to (0,+\infty)$ a non increasing mapping, define 
\begin{equation}
\label{sg}
s_g=\inf\left\{s\geq 0 :\sum_{k\geq 0}\sum_{\underline{i}\in\Lambda^k}k\left(\vert f_{\underline{i}}(K)\vert g(k)\right)^s<+\infty\right\}.
\end{equation} 
If \eqref{condibaker} is satisfied and $\dim (\mu)=\dim(S) ,$ one has
\begin{equation}
\begin{cases}\dim_H\Big(\limsup_{\underline{i}\in\Lambda^{*}}B\left(f_{\underline{i}}(x),\big(\vert f_{ \underline{i}}(K)\vert g(\vert \underline{i}\vert)\big)^{\frac{\delta s_g}{\dim(S)}}\right)\Big)=\dim(S) \text{ if }0<\delta\leq 1 \\
\dim_H\Big(\limsup_{\underline{i}\in\Lambda^{*}}B\left(f_{\underline{i}}(x),\big(\vert f_{ \underline{i}}(K)\vert g(\vert \underline{i}\vert)\big)^{\frac{\delta s_g}{\dim(S)}}\right)\Big)=\frac{ \dim(S) }{\delta}\text{ if }\delta\geq 1.

\end{cases}
\end{equation}

\end{theoreme}

\begin{remarque}
As mentioned in introduction, $\dim_H(\mu)=\dim(S)$ holds in many situations. For instance, any self-similar IFS acting on $ \mathbb{R} ^d$ with similarity dimension less than $d$ and satisfying Hochman's separation and additional irreducibility conditions \cite{Hoc} satisfies this property.  
\end{remarque}
\sk
In the next section, Section \ref{sec-ubi}, one proves a ubiquity theorem for self-conformal measures. In Section \ref{secprshrtag}, Theorem \ref{SHRTARG} is established using this ubiquity theorem. 

The last section, Section \ref{seccompbaker} is dedicated to the proof of Theorem \ref{compbaker}.

 \section{Mass transference principle and self-conformal measures}
\label{sec-ubi}
The key geometric notion developed in \cite{ED3} to handle inhomogeneous mass transference principles is the following.  
\begin{definition} 
\label{mucont}
{Let $\mu \in\mathcal{M}(\R^d)$, and $s\geq 0$.
The $s$-dimensional $\mu$-essential Hausdorff content of a set $A\subset \mathcal B(\R^d)$ is defined as}
{\begin{equation}
\label{eqmucont}
 \mathcal{H}^{\mu,s}_{\infty}(A)=\inf\left\{\mathcal{H}^{s}_{\infty}(E): \ E\subset  A , \ \mu(E)=\mu(A)\right\}.
 \end{equation}}
\end{definition}
As in the self-similar case, treated in \cite{ED3}, precise estimates of $ \mathcal{H}^{\mu,s}_{\infty}(A)$ are established when $\mu$ is a $C^{1}$ self-conformal measure in Theorem \ref{contss}.

\medskip


We will need the following notion of asymptotically covering sequences of balls, developed in \cite{ED2} (and also used in \cite{ED3}), to establish the desired ubiquity theorem.

\mk

 \begin{definition} 
\label{ac}
Let   $\mu\in \mathcal{M}(\R^d)$. The  sequence $\mathcal{B}= (B_n)_{n\in\mathbb{N}}$ of closed balls of $\R^d$ satisfying $\vert B_n \vert\to 0$  is said to be $\mu$-asymptotically covering (in short,  $\mu$-a.c) when  there exists  a constant $C>0$ such that for every open set $\Omega\subset \R^d $ and $g\in\mathbb{N}$, there is an integer  $N_\Omega \in\mathbb{N}$ as well  as $g\leq n_1 \leq ...\leq n_{N_\Omega}$ such that: 
\begin{itemize}
\item [(i)]$\forall \, 1\leq i\leq N_\Omega$, $B_{n_i}\subset \Omega$;
\item [(ii)]$\forall \, 1\leq i\neq j\leq N_\Omega$, $B_{n_i}\cap B_{n_j}=\emptyset$;
\item  [(iii)] also,
\begin{equation}
\label{majac}
\mu\Big(\bigcup_{i=1}^{N_\Omega}B_{n_i} \Big )\geq C\mu(\Omega).
\end{equation}
\end{itemize}
\end{definition}

 The following lemma is proved in \cite{ED2}, the second item will be used to apply our main theorem to self-conformal measures. For more details about this notion, derived from a covering property proved in the KGB-Lemma in \cite{BV}, one refers to \cite{ED2}.

\begin{lemme} 
\label{equiac}
Let   $\mu\in\mathcal{M}(\R^d)$  and $\mathcal{B} =(B_n :=B(x_n ,r_n))_{n\in\mathbb{N}}$ be a sequence of balls of  $\R^d$ with $\lim_{n\to +\infty} r_{n}= 0$.
\begin{enumerate}
\smallskip
\item
If $\mathcal{B} $ is $\mu$-a.c, then $\mu(\limsup_{n\rightarrow+\infty}B_n)=1.$
\smallskip
\item
 If there exists $v<1$ such that $ \mu \big(\limsup_{n\rightarrow+\infty}(v B_n) \big)=1$, then $\mathcal{B} $ is $\mu$-a.c. 
 \item If $\mu$ is doubling then $ \mu \big(\limsup_{n\rightarrow+\infty}(v B_n) \big)=1$  $\Leftrightarrow$ $\mathcal{B} $ is $\mu$-a.c.  
\end{enumerate}
\end{lemme}
  
One now recall the ubiquity mentioned above.

 \begin{theoreme}[\cite{ED3}] 
\label{lowani}
Let $\mu\in\mathcal{M}(\mathbb{R}^d),$  $\mathcal{B} =(B_n)_{n\in \N} $ be a $\mu$-a.c  sequence of closed balls of $\R^d$ such that $\vert B_n \vert \to 0$ and $\mathcal{U}=(U_n)_{n\in\mathbb{N}}$ a sequence of open sets such that $U_n \subset B_n$ for all $n\in\mathbb{N}$. Let $0\leq s<\underline{\dim}_H (\mu)$ such that for every $n$ large enough, $\mathcal{H}^{\mu,s}_{\infty}(U_n)\geq  \mu(B_n).$

Then
 \begin{equation}
\label{conc2ani}
 \dim_{H}\left(\limsup_{n\rightarrow +\infty}U_n\right)\geq s.
 \end{equation}
\end{theoreme}
In order to apply Theorem \ref{lowani}, precise estimates of essential contents of open sets  must be achieved. The next-subsection is dedicated to this problem when the measure is self-conformal and in the last sub-section of Section \ref{sec-ubi}, the desired ubiquity theorem is established.

\subsection{Geometric and dimensional properties of $C^{1}$ weakly conformal IFS}
\label{sec-geo}

In this subsection, we establish some basic properties of $C^1$ weakly conformal IFS's. We also prove basic dimension properties, which will be useful in the proof of Theorem \ref{SHRTARG} and we prove that weakly conformal IFS's satisfying the asymptotically weak separation condition   (Definition \ref{separacond}) with no exact overlaps satisfies the hypothesis of Theorem \ref{SHRTARG}.

Let  $m \geq 2$ be an integer.  In this section, one collects some useful geometric results when dealing with $C^{1}$ weakly conformal IFS.

In the rest if the article, the following notations will be used:
\begin{itemize}
\item $\Lambda(S)=\left\{1,...,m\right\}$ and $\Lambda(S)^{*}=\bigcup_{k\geq 0}\Lambda(S)^k$. When there is no ambiguity on the system $S$ involved, on will simply write $\Lambda(S)=\Lambda.$\sk
\item $K_{S}$ denotes the attractor of $S$ (or simply $K$ when the context is clear).\sk
\item For $\underline{i}=(i_1,...,i_k)\in \Lambda^{k}$, the cylinder $[\underline{i}]$ is defined by $$[\underline{i}]=\left\{(i_1,...,i_k,x_1,x_2,...):(x_1,x_2,...)\in\Lambda^{\mathbb{N}}\right\}.$$

 Moreover, if $(\alpha_n)_{n\in\mathbb{N}}$ is a sequence of real numbers, one sets
$$\alpha_{\underline{i}}=\alpha_{i_1}\times ...\times \alpha_{i_k}$$
and
$$f_{\underline{i}}=f_{i_1}\circ...\circ f_{i_k}.$$
For example, given the probability vector $(p_1,..,p_m),$ $p_{\underline{i}}=p_{i_1}\times...\times p_{i_k}.$\sk
\item The set $\Lambda^{\mathbb{N}}$ will always be endowed with the topology generated by the cylinders. The set of  probability measures on the Borel sets with respect to this topology will be denoted $\mathcal{M}(\Lambda^{\mathbb{N}})$.\sk 
 
 \item The shift operation on $\Lambda^{\mathbb{N}}$ is denoted by $\sigma$ and defined for any $(i_1,i_2,...)\in\Lambda^{\mathbb{N}}$ by 
 \begin{equation}
 \label{defshift}
  \sigma((i_1,i_2,...))=(i_2,i_3,...).
 \end{equation}
 \item The canonical projection of $\Lambda^{\mathbb{N}}$ on $K$ will be denoted $\pi_{ \Lambda}$ (or simply $\pi$ when there is no ambiguity) and, fixing $x\in K,$ is defined, for any $(i_1,i_2,....)\in\Lambda^{\mathbb{N}}$, by 
 \begin{equation}
 \label{canproj}
 \pi((i_1,...))=\lim_{k\rightarrow+\infty}f_{i_1}\circ...\circ f_{i_k}(x).
 \end{equation}
 It is easily verified that $\pi$ is independent of the choice of $x.$
\end{itemize}

Consider $S=\left\{f_1 ,...,f_m\right\}$ a $C^{1}$ weakly conformal IFS of attractor $K$ and, for every $x\in K$, $k\in\mathbb{N}$  and $\underline{i}=(i_1,..,i_k)\in\Lambda^k$, write 
$$c_{\underline{i}}(x)=\vert\vert f_{\underline{i}}'(x)\vert\vert.$$

Let us recall the following result established as \cite[Lemma 5.4]{FH}. 
\begin{lemme}[\cite{FH}]
\label{lemmewef}
For any $c>1$, there exists a constant $D(c)$ such that, for every $k\in\mathbb{N}$, for every $\underline{i}\in \Lambda^k$ and every $x,y\in K$,
\begin{equation}
\label{bdist1}
D(c)^{-1}c^{-k} \vert \vert f_{\underline{i}}^{\prime}(x)\vert\vert\cdot\vert\vert x-y\vert\vert\leq  \vert \vert f_{\underline{i}}(x)-f_{\underline{i}}(y)\vert\vert\leq D(c)c^k \vert \vert f_{\underline{i}}^{\prime}(x)\vert\vert\cdot\vert\vert x-y\vert\vert,
\end{equation}

\begin{equation}
\label{bdist2}
D(c)^{-1}c^{-k} \vert \vert f_{\underline{i}}^{\prime}(x)\vert\vert\leq  \vert f_{\underline{i}}(K)\vert\leq D(c)c^k \vert \vert f_{\underline{i}}^{\prime}(x)\vert\vert ,
\end{equation}

\end{lemme}
\begin{remark}
\label{remX}
Let $X\subset U$ be a compact set. Equation \eqref{bdist1} actually holds for any $(x,y)\in X^2.$.
\end{remark}

Note that, for every $k\in\mathbb{N},$ writing $\chi=\pm$, one has
\begin{align*}
c^{ \chi k}\vert\vert f_{\underline{i}}^{\prime}(x)\vert\vert =\vert\vert f_{\underline{i}}^{\prime}(x)\vert\vert^{1+\frac{\chi k\log c}{\log \vert \vert f_{ \underline{i}} ^{\prime}(x)\vert\vert}}.
\end{align*}

Moreover there exists two constants $0< t_1 \leq t_2$ such that $t_1\leq \frac{k}{\log \vert \vert f_{ \underline{i}} ^{\prime}(x)\vert\vert}\leq t_2 .$ Combining this with Lemma \ref{lemmewef}, for any $\theta>0$, there exists $\widetilde{C}_{\theta}>0$, such that for every $k\in\mathbb{N}$, every $\underline{i}\in\Lambda^k$ and every  $x,y\in K$,

\begin{equation}
\label{equreg}
\widetilde{C}^{-1}_{\theta}c_{\underline{i}}(x)^{1+\theta}\vert \vert x-y\vert\vert\leq \vert \vert f_{\underline{i}}(x)-f_{\underline{i}}(y)\vert\vert \leq \widetilde{C}_{\theta}c_{\underline{i}}(x)^{1-\theta}\vert \vert x-y\vert\vert.
\end{equation} 

In particular, there also exists $\widehat{C}_{\theta}$  for every $\underline{i}\in \Lambda^{*}$ and every $x\in K,$ one has
\begin{equation}
\label{equaxibar}
 \widehat{C}_{\theta}^{-1}c_{\underline{i}}^{1+\theta}(x)\vert K\vert\leq \vert f_{ \underline{i}}(K) \vert\leq \widehat{C}_{\theta}c_{\underline{i}}^{1-\theta}(x)\vert K\vert.
 \end{equation} 
Let us remark also that \eqref{equaxibar} also implies that there exists $0<\alpha \leq \beta<1$ as well as $C_{\alpha},C_{\beta}>0$ such that, for any $k\in\mathbb{N},$ 
\begin{equation}
\label{encafibar}
C_{\alpha}\alpha^k \leq \vert f_{\underline{i}}(K)\vert\leq C_{\beta}\beta^{k}.
\end{equation}

\subsubsection{Lyapunov exponent of $C^{1}$ weakly conformal IFS's}
Let $m\geq 2$ and let us fix a $C^1$ IFS, $S=\left\{f_1,...,f_m\right\}.$

\begin{proposition}[\cite{FH}]
\label{expolyapu}
For weakly conformal systems, for any $x=(x_n)_{n\in\mathbb{N}},$ the Lyapunov exponent is well defined
\begin{equation}
\lambda(x)=-\lim_{n\rightarrow+\infty}\frac{\log \vert f_{x_1}\circ ...\circ f_{x_n}(K) \vert}{n}.
\end{equation} 
\label{contilyapu}
Moreover, for any probability vector $(p_1,...,p_m)\in [0,1]^m$, denoting $\nu\in\mathcal{M}(\Lambda^{\mathbb{N}})$ the measure defined by $\nu([\underline{i}])=p_{\underline{i}},$ then there exists $\lambda_{\nu}\geq 0$ such that for $\nu$-almost any $x=(x_n)_{n \in\mathbb{N}}$, 
\begin{equation}
\label{defexpolyap}
\lambda(x)=\int \lambda(y)d\nu(y):=\lambda_{\nu}.
\end{equation}
\begin{remark}
\label{remboundlyapu}
By \eqref{encafibar}, the Lyapunov exponent are uniformly (with respect to weakly conformal measures) bounded by above and below by some positive constant.
\end{remark}
\end{proposition}
\begin{corollary}
\label{convproba}
Let $((p_1 ^{(k)},...,p_m ^{(k)}))_{k\in \mathbb{N}}\in \left([0,1]^m\right)^{\mathbb{N}}$  be a sequence of probability vectors such that $(p_1^{(k)},...,p_m^{(k)})\to(p_1,...,p_m).$ Denote  for $k\in\mathbb{N}$ $\nu,\nu_{k}\in\mathcal{M}(\Lambda^{\mathbb{N}})$ the measures defined by, for any cylinder $[(i_1,...i_n)],$ $$\nu_{k}([(i_1,...,i_k)])=p_{i_1}^{(k)} \cdot ...\cdot p_{i_n}^{(k)}\text{ and } \nu([(i_1,...,i_n)])=p_{i_1} \cdot ...\cdot p_{i_n}.$$ Then $\nu_k \underset{k\to +\infty}{\rightarrow} \nu$ weakly, so that
$$\lim_{k\rightarrow+\infty}\lambda_{\nu_k}= \lambda_{\nu}. $$
\end{corollary}
\subsubsection{Dimension of weakly-conformal IFS's}
Let us recall the following fundamental result.

\begin{theoreme}[Feng-Hu, \cite{FH}]
\label{decompofeng}

Let $(p_1,...,p_m)\in [0,1]^m$ be a probability vector, $\nu\in\mathcal{M}(\Lambda^{\mathbb{N}})$ defined by, for any $\underline{i}\in\Lambda^*,$ $\nu([\underline{i}])=p_{\underline{i}}$ and $\mu=\nu\circ \pi^{-1}.$ 

There exists $h\geq 0$ such that for $\mu$-almost every $x\in K$, there exists $ \mu^{\pi^{-1}(\left\{x\right\})}\in\mathcal{M}( \Lambda^{\mathbb{N}})$ such that:
\begin{itemize}
\item[•] $\mu^{\pi^{-1}(\left\{x\right\})}(\pi^{-1}(\left\{x\right\}))=1.$\mk
\item[•] for $\mu^{\pi^{-1}(\left\{x\right\})}$-almost $y=(y_1,...,y_n,..)$, 
\begin{equation}
\label{exadimmupi}
\frac{-\log \mu^{\pi^{-1}(\left\{x\right\})}([y_1,...,y_n])}{n}\to h.
\end{equation}
\item[•] for every Borel set $A\subset \Lambda^{\mathbb{N}},$
\begin{equation}
\nu(A)=\int_{K}\mu^{\pi^{-1}(\left\{x\right\})}(A)d\mu(x).
\end{equation} 
\item[•] denoting $\lambda=-\sum_{1\leq i\leq m}p_i \log c_i,$ $\mu$ is exact-dimensional (Definition \ref{dimmu}) and 
$$\dim(\mu)=\frac{-h-\sum_{1\leq i\leq m}p_i\log p_i}{\lambda} .$$
\end{itemize}
\end{theoreme}

\sk

%
%
Proposition \ref{propopres} is now proved.

\begin{proof}
Assume first that the limit exists in $\mathbb{R}\cup\left\{-\infty\right\}$ and let us show that it is independent of the choice of $z$ and that the limit is $>-\infty$. Let  $ c>1$ be a a real number. By \eqref{lemmewef}, following the notation involved, for any $k\in\mathbb{N},$ one has
\begin{align}
\label{poiuy}
\log\left(\sum_{\underline{i}\Lambda^k}D(c)^{-s}c^{-sk}\vert f_{\underline{i}}(K)\vert ^s\right)\leq \log\left(\sum_{\underline{i}\in \Lambda^k}\vert\vert f_{\underline{i}}' (z)\vert\vert^s\right)\leq \log\left(\sum_{\underline{i}\Lambda^k}D(c)^s c^{sk}\vert f_{\underline{i}}(K)\vert^s \right).
\end{align}
Since  \eqref{poiuy} holds for any $c>1$, one gets that 
\begin{equation}
\label{indepz}
\lim_{k\rightarrow+\infty}\frac{1}{k}\Big(\log\left( \sum_{\underline{i}\in \Lambda^k}\vert \vert f_{\underline{i}}^{\prime}(z)\vert \vert^s \right)-\log\left(\sum_{\underline{i}\Lambda^k}\vert f_{\underline{i}}(K)\vert ^s\right)\Big)=0,
\end{equation}
which proves  that this quantity does not depend on $z$. Moreover, there exists $b>0$ so that for any $k\in\mathbb{N}$, any $\underline{i}\in\Lambda^k$, any $x\in K$, $$\vert \vert f_{\underline{i}}'(x)\vert\vert\geq b^{k}.$$
This implies that if $P_z(s)$ is well defined, then $P_z(s)>-\infty.$

Let us now prove that the limit exists. For $k\in\mathbb{N},$  write
\begin{equation}
\label{defg}
g_n =\log\left(\sum_{\underline{i}\in \Lambda^k}\vert f_{\underline{i}}(K)\vert^s \right).
\end{equation}

\begin{lemme}
\label{lemmutpreuv}
For any $\varepsilon>0,$ there exists a constant $M_{\varepsilon}>0$ such that for any $n,m\in\mathbb{N},$ one has
\begin{equation}
\label{quasisubadd}
g_{n+m}\leq M_{ \varepsilon}+m\varepsilon+g_n +g_m.
\end{equation}
Furthermore, any sequence $(g_n)_{n\in\mathbb{N}}$ verifying \eqref{quasisubadd} is such that $(\frac{g_n}{n})_{n\in\mathbb{N}}$ converges in $\mathbb{R}\cup\left\{-\infty\right\}$.
\end{lemme}
\begin{proof}
Let us start by proving the second statement. Let $(g_n)_{n\in\mathbb{N}}$ be a sequence satisfying \eqref{quasisubadd}. Fix $\varepsilon>0$ and $M_{\varepsilon}$ satisfying \eqref{quasisubadd}.  For any $q\in\mathbb{N},$  $b\in\mathbb{N},$ $0\leq r< q$, one has
\begin{align*}
&g_{bq+r}\leq bg_{q}+g_r +(bq+r) \varepsilon+(b+1)M,
\\& \Rightarrow \frac{g_{bq+r}}{bq+r}\leq \frac{bq}{bq+r}\cdot \frac{g_q}{q}+\frac{(b+1)M+g_r}{bq+r}+\varepsilon. 
\end{align*}
 Fixing $q$ large enough so that $\frac{(b+1)M}{bq}\leq \varepsilon$, for any large $b\in\mathbb{N}$, one has 
 $$\frac{g_{bq+r}}{bq+r}\leq (1+\varepsilon)\frac{g_q}{q}+2\varepsilon.$$
 This implies that
  $$\limsup_{n\rightarrow+\infty}\frac{g_n}{n}\leq (1+\varepsilon)\liminf_{n\rightarrow+\infty}\frac{g_n}{n}+\varepsilon.$$
  Letting $\varepsilon\to 0$ proves the statement. 

One now shows that $g_n$ satisfies \eqref{quasisubadd}.

 Let $k\in\mathbb{N}$ and $\underline{i}\in\Lambda^k.$ Let us begin by the following lemma.
\begin{lemme}
\label{subaddf}
Following the notation of Lemma \ref{lemmewef}, one has, for any $ \underline{j}\in\Lambda^{*},$
\begin{equation}
\label{bdist3}
\frac{1}{2}D(c)^{-2}c^{-2k} \vert f_{\underline{i}}(K)\vert\cdot \vert f_{\underline{j}}(K) \vert\leq  \vert f_{\underline{i}\underline{j}}(K))\vert\leq 2D(c)^{2}c^{2k} \vert f_{\underline{i}}(K)\vert\cdot \vert f_{\underline{j}}(K) \vert.
\end{equation}
\end{lemme}
\begin{proof}
Let us start by establishing the lower-bound. 

Let $x,y\in K$ such that 
\begin{equation}
\label{boa}
\vert\vert f_{ \underline{j}}(x)-f_{\underline{j}}(y)\vert\vert\leq\vert f_{\underline{j}}(K)\vert\leq 2\vert\vert f_{\underline{j}}(x)-f_{\underline{j}}(y)\vert\vert. 
\end{equation}
By Lemma \ref{lemmewef}, one has 
\begin{align}
\label{boa1}
D(c)^{-1}c ^{-k}\vert\vert f_{\underline{i}}^{\prime}(f_{\underline{j}}(x))\vert\vert\cdot \vert\vert f_{\underline{j}}(x)-f_{\underline{j}}(y)\vert\vert\leq \vert \vert f_{\underline{i}\underline{j}}(x)-f_{\underline{i}\underline{j}}(y)\vert\vert\leq \vert f_{\underline{i}\underline{j}}(K)\vert,
\end{align}
and
\begin{align}
\label{boa2}
\vert\vert f_{\underline{i}}^{\prime}(f_{\underline{j}}(x))\vert\vert\geq D(c)^{-1}c^{-k}\vert f_{\underline{i}}(K)\vert.
\end{align} 
Combining \eqref{boa}, \eqref{boa1} and \eqref{boa2}, one obtains
\begin{align*}
\frac{1}{2}D(c)^{-2}c^{-2k}\vert f_{\underline{i}}(K)\vert \cdot \vert f_{\underline{j}}(K)\vert\leq \vert f_{\underline{i}\underline{j}}(K)\vert.
\end{align*}

Let us focus now on the upper-bound. Let $x,y\in K$ such that 
\begin{equation}
\label{boa5}
\vert\vert f_{\underline{i}\underline{j}}(x)-f_{\underline{i}\underline{j}}(y)\vert\vert\geq \frac{1}{2}\vert f_{\underline{i}\underline{j}}(K)\vert.
\end{equation}
Using again Lemma \ref{lemmewef}, one has
\begin{align}
\label{boa6}
\vert\vert f_{\underline{i}\underline{j}}(x)-f_{\underline{i}\underline{j}}(y)\vert\vert\leq D(c)c^k \vert\vert f_{\underline{i}}'(f_{\underline{j}}(x))\vert\vert \cdot \vert\vert f_{ \underline{j}}(x)-f_{\underline{j}}(y) \vert\vert\leq D(c)^2 c^{2k}\vert f_{\underline{i}}(K)\vert \cdot \vert f_{\underline{j}}(K)\vert.
\end{align}
The upper-bound is obtained by combining \eqref{boa5} and \eqref{boa6}.
\end{proof}
By Lemma \ref{subaddf}, for any $c>1$ and any $n,n^{\prime}\in\mathbb{N},$ one has
\begin{align*}
&\log\left(\sum_{\underline{i}\in\Lambda^{n+n'}}\vert f_{\underline{i}}(K) \vert^s\right)=\log\left(\sum_{\underline{i}\in \Lambda^n,\underline{j}\in\Lambda^{n'}}\vert f_{\underline{i}\underline{j}}(K)\vert^s\right)\\
&\leq \log\left(\sum_{\underline{i}\in \Lambda^n,\underline{j}\in\Lambda^{n'}}2^s D(c)^{2s}c^{2sn}\vert f_{\underline{i}}(K)\vert^s \vert f_{\underline{j}}(K)\vert^s\right)\\
&=n \cdot 2s\log(c)+\log(2^s D(c)^{2s})+\log\left((\sum_{\underline{i}\in\Lambda^n}\vert f_{\underline{i}}(K)\vert^s)\times(\sum_{\underline{j}\in\Lambda^{n^{\prime}}}\vert f_{\underline{j}}(K)\vert^s)\right)\\
&\leq 2sn\log(c)+\log(2^s D(c)^{2s})+g_{n}+g_{n'}.
\end{align*} 
This concludes the proof of Lemma  \ref{lemmutpreuv}.

\end{proof} 
 Lemma \ref{lemmutpreuv} together with \eqref{indepz} concludes the proof.
\end{proof}

Since $P_z(s) $ does not depend on $z$, one writes 
$$P_z (s)=P(s)=\lim_{k\rightarrow+\infty}\frac{1}{k}\log\left(\sum_{\underline{i}\Lambda^k}\vert f_{\underline{i}}(K)\vert ^s\right)$$

%
%
%
%


\subsubsection{A class of IFS's satisfying $\dim(S)=\dim_H(K)$} 
In this section it is proved that weakly conformal IFS's satisfying the asymptotically weak separation condition (AWSC in short) with no exact overlaps satisfies $\dim(S)=\dim_H(K),$ so that Theorem \ref{SHRTARG} applies for those IFS's. 

Let us first introduce, for all $k\in\mathbb{N}$, 
\begin{equation}
\label{deflam(k)}
\Lambda^{(k)}=\left\{\underline{i}=(i_1,...,i_n) \in\Lambda^{*}:\vert f_{\underline{i}}(K)\vert \leq 2^{-k}<\vert f_{(i_1,...i_{k-1})}(K)\vert\right\}.
\end{equation}

\begin{definition}
\label{separacond}
Let $m\geq 2$ and $S=\left\{f_1,...,f_m\right\}$ a weakly conformal IFS.    For $k\in\mathbb{N}$, define
\begin{equation}
\label{deftk}
t_k(S)=\max_{x\in\mathbb{R}^d}\#\left\{f_{\underline{i}}:\underline{i}\in\Lambda^{(k)}\text{ and } f_{\underline{i}}(K)\cap B(x,r)\neq \emptyset\right\},
\end{equation}
The system $S$  satisfies the asymptotically weak separation condition (AWSC) \cite{Feng2007} when  
$$\frac{\log t_k}{k}\to 0 .$$

\end{definition}

Let us recall the notion of dimension regular weakly conformal IFS's, introduced by Barral and Feng in \cite{BFmult} in the case of self-similar IFS's.

\begin{definition}[ \cite{BFmult}]
\label{dimreg}
One says that $S$ is dimension regular  if for any weakly conformal measure $\mu\in\mathcal{M}(\mathbb{R}^d)$ associated with the probability vector $(p_1,...,p_m)\in[0,1]^m$ and $S$, recalling \eqref{expolyapu} and denoting $\nu\in\mathcal{M}(\Lambda^{\mathbb{N}})$ verifying $\mu=\nu\circ\pi^{-1}$, one has
\begin{equation}
\dim(\mu)=\min\left\{\frac{-\sum_{1\leq i\leq m}p_i \log(p_i)}{\lambda_{\nu}}, d\right\},
\end{equation}
where $\lambda_{\nu}$ is defined by \eqref{defexpolyap}.
\end{definition}
\begin{remark}
When $S$ is self-similar, calling $0<c_1 ,...,c_m<1$ the contraction ration of the similarities $f_1,...,f_m$, for any probability vector $(p_1,...,p_m),$ $\mu$ and $\nu$ as in Definition  \ref{dimreg}, one has
\begin{equation}
\dim(\mu)=\min\left\{\frac{-\sum_{1\leq i\leq m}p_i \log(p_i)}{\lambda_{\nu}}, d\right\}=
\min\left\{\frac{\sum_{1\leq i\leq m}p_i \log(p_i)}{\sum_{i=1}^m p_i \log (c_i )}, d\right\}.
\end{equation}
\end{remark}

\begin{proposition}
\label{dimweakconf}
 Assume that  $S=\left\{f_1,...,f_m\right\}$   satisfies the AWSC without exact overlaps. Then $S$ is dimension regular. Moreover, $\dim(S)=\dim_H(K).$  
 \end{proposition}

Before proving Proposition \ref{dimweakconf}, let us start by the following lemma.

\begin{lemme}
\label{weakconfapeupres}
Let $\varepsilon>0$ and $s\geq 0$ be a real numbers. There exists $k\in\mathbb{N}$, a probability vector $(p_{\underline{i}})_{\underline{i}\in \Lambda^k}$  such that the weakly conformal measure $\nu$ associated with $S'=\left\{f_{\underline{i}}\right\}_{\underline{i}\in\Lambda^k}$ and $(p_{\underline{i}})_{\underline{i}\in\Lambda^k}$ verifies, for any $p\in\mathbb{N}$ and $\underline{i}_1 ,...,\underline{i}_p \in\Lambda^{k},$ 
\begin{equation}
e^{-kp\varepsilon}\frac{\vert f_{\underline{i}_1 ... \underline{i}_p}(K)\vert^s}{e^{pkP(s)}}\leq \nu([\underline{i}_1 ... \underline{i}_p])\leq e^{kp\varepsilon}\frac{\vert f_{\underline{i}_1 ... \underline{i}_p}(K)\vert^s}{e^{pkP(s)}}
\end{equation} 
\end{lemme}
\begin{proof}
Fix $\varepsilon>0$ and $c>1$ such that $8s\log c\leq \varepsilon.$

By Lemma \ref{lemmutpreuv} there exists $k\in\mathbb{N}$ large enough so that, the constant named $D(c)$ in Lemma \ref{lemmewef} verifies $ \frac{\log D(c)}{k}\leq \log c $ and  
\begin{equation}
\label{aqs3}
\vert \frac{1}{k}\log \sum_{\underline{i}\in\Lambda^k}\vert f_{\underline{i}}(K)\vert^s -P(s)\vert\leq \frac{\varepsilon}{2}.
\end{equation}

 Writing again $g_k =\log \sum_{\underline{i}\in\Lambda^k}\vert f_{\underline{i}}(K)\vert^s,$ let us define the probability vector $(p_{\underline{i}})_{\underline{i}\in\Lambda^k}$ by setting
 $$p_{\underline{i}}=\frac{\vert f_{\underline{i}}(K)\vert^s}{e^{g_k}}.$$
 Let $\nu$ be the weakly conformal measure associated with $S'=\left\{f_{\underline{i}}\right\}_{\underline{i}\in\Lambda^k}$ and $(p_{\underline{i}})_{\underline{i}\in\Lambda^k}.$ Applying Lemma \ref{bdist3}, for any $p\in\mathbb{N},$ $\underline{i}_1,...,\underline{i}_p$, 
 \begin{align}
 \label{aqs0}
 D(c)^{-2p}c^{-2kp}\leq \frac{\vert f_{\underline{i}_1}\circ ...\circ f_{\underline{i}_p}(K)\vert}{\prod_{j=1}^p \vert f_{\underline{i}_j}(K)\vert}\leq  D(c)^{2p}c^{2kp}.
 \end{align}
Also
\begin{align}
\label{aqs1}
D(c)^{2sp} c^{2skp}=e^{pk 2s\cdot (\frac{\log D(c)}{k}+\log c)}\leq e^{\frac{\varepsilon}{2} pk}.
\end{align}
As a consequence, for any $p\in\mathbb{N}$ and any $\underline{i}_1,...,\underline{i}_p \in\Lambda^k,$ one has
\begin{align*}
\nu([\underline{i}_1 ... \underline{i}_p])=p_{\underline{i}_1}\cdot ...\cdot p_{\underline{i}_p}=\frac{\prod_{j=1}^p \vert f_{\underline{i}_j}(K)\vert^s}{e^{pg_k}}=\frac{\prod_{j=1}^p \vert f_{\underline{i}_j}(K)\vert^s}{e^{kp(\frac{g_k}{k}-P(s))}e^{pkP(s)}}.
\end{align*}
Using \eqref{aqs3}, \eqref{aqs0} and \eqref{aqs1} concludes the proof.

\end{proof}

\begin{remark}
\label{frakegibbs}
The measure $\nu$ can be extended over $\Lambda^{\mathbb{N}}$ by the usual arguments. Moreover, for any $\underline{i}=(i_1,...,i_n)\in\Lambda^{*},$ write $n_1=k\lfloor \frac{n}{k} \rfloor$ and $n_2=k(\lfloor \frac{n}{k} \rfloor+1).$ Consider $\underline{j}\in \Lambda^{n_1}$ such that $[\underline{i}]\subset [\underline{j}]$ and $\underline{\ell}=(\ell_1,...,\ell_{n_2-n})\in\Lambda^{n_2-n},$ one has 
\begin{equation}
e^{-n_2\varepsilon}\frac{\vert f_{(i_1, ..., i_{n},\ell_1,...,\ell_{n_2-n})}(K)\vert^s}{e^{n_2 P(s)}}\leq\nu([\underline{i}\underline{\ell}]) \leq \nu([\underline{i}])\leq \nu([\underline{j}])\leq e^{n_1\varepsilon}\frac{\vert f_{(i_1,...,i_{n_1})}(K)\vert^s}{e^{n_1 P(s)}}.
\end{equation} 
Since, by Lemma \ref{bdist3}, there exists a constant $C>0$ such that, uniformly on $\underline{i},\underline{j},\underline{i}\underline{\ell},$ one has
\begin{equation*}
C^{-1}\leq\min\left\{\frac{\vert f_{\underline{j}}(K)\vert}{\vert f_{\underline{i}}(K)\vert},\frac{\vert f_{\underline{i}}(K)\vert}{\vert f_{\underline{i}\underline{\ell}}(K)\vert}\right\} \leq\max\left\{\frac{\vert f_{\underline{j}}(K)\vert}{\vert f_{\underline{i}}(K)\vert},\frac{\vert f_{\underline{i}}(K)\vert}{\vert f_{\underline{i}\underline{\ell}}(K)\vert}\right\}\leq C,
\end{equation*}
there exists a constant $\gamma_{s,\varepsilon}$, such that, for any $\underline{i}=(i_1,...,i_n)\in\Lambda^{*}$, one has  
\begin{equation}
\gamma_{s,\varepsilon}^{-1}e^{-n\varepsilon}\frac{\vert f_{\underline{i}}(K)\vert^s}{e^{nP(s)}}\leq \nu([\underline{i}])\leq \gamma_{s,\varepsilon}e^{n\varepsilon}\frac{\vert f_{\underline{i}}(K)\vert^s}{e^{nP(s)}}.
\end{equation}

\end{remark}
Let  us now prove Proposition \ref{dimweakconf}.

\begin{proof}
Call $K$ the attractor of $S$. Let us show first that if any system $S$ satisfying the AWSC also verifies that, for any weakly-conformal measure  $\mu\in\mathcal{M}(\mathbb{R}^d)$ associated with a probability vector $(p_1,...,p_m)$ and $S$, 
\begin{equation}
\label{weakconfdimreg}
\dim(\mu)=\frac{-\sum_{1\leq i\leq m}p_i \log p_i}{\lambda_{\nu}},
 \end{equation}
where $\nu$ is the measure associated on $\Lambda^{\mathbb{N}},$ then $\dim(S)=\dim_H (K).$ 

Fix $\varepsilon>0$ consider $k\in\mathbb{N}$, $S'=\left\{f_{\underline{i}}\right\}_{\underline{i}\in\Lambda^k}$ and $\nu$ as in Lemma \ref{weakconfapeupres} applied with $s=\dim(S).$ Note that, since $S$ satisfies the AWSC, so does $S'$. Then, considering the measure $\mu=\nu\circ \pi^{-1}$, where $\pi$ is the canonical projection, one has 
\begin{align*}
\dim(S)-\varepsilon\leq \dim (\mu)=\frac{-\sum_{\underline{i}\in\Lambda^k}p_{\underline{i}}\log p_{\underline{i}}}{\lambda_{\nu}}\leq \dim(S)+\varepsilon.
\end{align*}   
This proves that $\dim_H (K)\geq \dim (S)-\varepsilon.$ Since it always holds that $\dim_H(K)\leq \dim(S)$ (see \cite{Fa1}) and $\varepsilon$ is arbitrary, 
$$\dim_H(K)=\dim(S).$$ 
 
 Let us now prove that, for any system satisfying the AWSC, \eqref{weakconfdimreg} holds for every weakly conformal measure $\mu$.

Let $\mu\in\mathcal{M}(\mathbb{R}^d)$ be a weakly conformal measure associated with $S$ and a probability vector $(p_1,...,p_m)$ and $\nu\in\mathcal{M}(\Lambda^{\mathbb{N}})$ such that $\mu=\nu\circ \pi^{-1}.$ 

One applies Theorem \ref{decompofeng} to $\mu.$

Moreover, it comes from from the proof of Theorem \ref{decompofeng} \cite{FH}, that for any $\varepsilon>0,$ for $\mu$-almost any $x\in K$ such that $\mu^{\pi^{-1}(\left\{x\right\})}$ exists and satisfies the two first items of Theorem \ref{decompofeng}, there exists $n_0$ large enough so that, for any $n\geq n_0$, there exists $\underline{i}_1,...,\underline{i}_{N_n}$ such that:
\begin{itemize} 
\item[•] for any $1\leq j\leq N_n,$
\begin{equation}
\label{equprop1}
e^{-n(\lambda+\varepsilon)}\leq\vert f_{\underline{i}_j}(K)\vert\leq e^{-n(\lambda-\varepsilon)},
\end{equation} 
 \item[•] one has
\begin{equation}
\label{equprop2}
\mu^{\pi^{-1}(\left\{x\right\})}\left(\bigcup_{1\leq j\leq N_n}[\underline{i}_j]\right)\geq \frac{1}{2},
\end{equation}
\item[•] for any $1\leq j\leq N_n,$
\begin{equation}
\label{equprop3}
e^{-n(h+\varepsilon)}\leq\mu^{\pi^{-1}(\left\{x\right\})}([\underline{i}_j])\leq e^{-n(h-\varepsilon)}
\end{equation}
\end{itemize}
Assume that $h>0$ and take $0<\varepsilon<\min\left\{ \frac{h}{2},\frac{ \lambda}{2}\right\}.$

Combining \eqref{equprop2} and \eqref{equprop3}, one gets
\begin{equation}
N_n \geq \frac{1}{2}e^{n(h-\varepsilon)}.
\end{equation}
Note that $\#\left\{k:e^{-n(\lambda+\varepsilon)}\leq 2^{-k}\leq e^{-n(\lambda-\varepsilon)}\right\}\leq\frac{2n\varepsilon}{\log 2}.$ As a consequence, there exists $k\in [\frac{n(\lambda-\varepsilon)}{\log 2},\frac{n(\lambda+\varepsilon)}{\log 2}]$ such that 
\begin{equation}
\# \Lambda^{(k)}\cap \left\{[\underline{i}_j]\right\}_{1\leq j\leq N_n}\geq \frac{N_n}{\frac{2n\varepsilon}{\log 2}}\geq \frac{\frac{1}{2}e^{\frac{nh}{2}}}{\frac{2n\varepsilon}{\log 2}}.
\end{equation}
 Since for any $\leq j \leq N_n,$ $[ \underline{i}_j]\cap \pi^{-1}(\left\{x\right\})\neq \emptyset,$ one also has $f_{\underline{i}_j}(K)\subset B(x,e^{-n(\lambda-\varepsilon)}),$
 so that, writing $n'=\lfloor \frac{n(\lambda-\varepsilon)}{\log2} \rfloor$, one has 
\begin{equation}
 \#\left\{\underline{i}\in \Lambda^{(n')}:f_{\underline{i}}(K)\cap B(x,2^{-n'})\right\}\geq \frac{\frac{1}{2}e^{\frac{nh}{2}}}{\frac{2n\varepsilon}{\log 2}}.
\end{equation}
In particular, recalling \eqref{deftk},
$$\frac{\log t_k}{k}\nrightarrow 0, $$
and $S$ does not satisfy the AWSC. As a consequence, $S$ satisfies the AWSC implies $h=0$, which, recalling the last item of Theorem \ref{decompofeng}, concludes the proof.
\end{proof}

\begin{remark}
If $S$ is a self-similar system and satisfies the OSC, then it satisfies the AWSC and has no exact overlaps, so that Theorem \ref{dimweakconf} holds for $S$.
\end{remark}

\subsection{Essential content for weakly conformal measures}

Estimates of essential contents for weakly conformal measures are now established.

\begin{theoreme}
\label{contss}
Let $S$ be a $C^{1}$ weakly conformal IFS of $\R^d$. 

Let $K$ be the attractor of $S$ and $\mu$ be a weakly conformal measure  associated with $S$. Then, 

\sk

 For any $0\leq s<\dim(\mu)$, for any $0<\varepsilon \leq \frac{1}{2}$, there exists a constant $c=c(d,\mu,s,\varepsilon)>0 $ depending on the dimension $d$, $\mu$, $s$ and $\varepsilon$  only, such that for any ball $B=B(x,r)$ centered on $K$ and $r\leq 1$, for any open set $\Omega$, one has 
\begin{align}
\label{genhcontss}
&c(d,\mu,s,\varepsilon)\vert B\vert ^{s+\varepsilon}\leq\mathcal{H}^{\mu, s }_{\infty}(\widering{B})\leq  \mathcal{H}^{ \mu, s}_{\infty}(B)\leq\vert B\vert ^{s} \nonumber\\
&
c(d,\mu,s,\varepsilon)\mathcal{H}^{s+\varepsilon}_{\infty}(\Omega \cap K)\leq \mathcal{H}^{\mu,s}_{\infty}(\Omega)\leq \mathcal{H}^{s}_{\infty}(\Omega \cap K).
\end{align}

\sk

 For any $s>\dim(\mu)$, $ \mathcal{H}^{\mu,s}_{\infty}(\Omega)=0.$

\end{theoreme}
\begin{remarque}
\item[•] The system $S$ is not assumed to verify any separation condition.\mk
\item[•] When the maps are similarities, one still has, for any $s>\dim(\mu)$, $ \mathcal{H}^{\mu,s}_{\infty}(\Omega)=0$ but for $s<\dim (\mu)$, there exists a constant $c(d,\mu,s)$ such that the following more precise estimates holds true \cite{ED3}:
\begin{align}
\label{genhcontsss}
&c(d,\mu,s)\vert B\vert ^{s}\leq\mathcal{H}^{\mu, s }_{\infty}(\widering{B})\leq  \mathcal{H}^{ \mu, s}_{\infty}(B)\leq\vert B\vert ^{s}\text{ and } \nonumber\\
&c(d,\mu,s)\mathcal{H}^{s}_{\infty}(\Omega \cap K)\leq \mathcal{H}^{\mu,s}_{\infty}(\Omega)\leq \mathcal{H}^{s}_{\infty}(\Omega \cap K).
\end{align}
For any $s>\dim(\mu)$, $ \mathcal{H}^{\mu,s}_{\infty}(\Omega)=0.$

\end{remarque}
\begin{proof}

The following modified version of Besicovitch covering lemma will be useful in this section.
\begin{proposition}[\cite{ED2}]
\label{besimodi}
For any $0<v\leq 1$, there exists $Q_{d,v} \in\mathbb{N}^{\star}$, a constant depending only on the dimension $d$ and $v$, such that for every $E\subset [0,1]^{d}$,  for every set $\mathcal{F}=\left\{B(x, r_{(x)} ): x\in E,  r_{(x)} >0 \right\}$, there exists $\mathcal{F}_1,...,\mathcal{F}_{Q_{d,v}}$ finite or countable sub-families of $\mathcal{F}$ such that:\medskip
\begin{itemize}

\item
$\forall 1\leq i\leq Q_{d,v}$, $\forall L\neq L'\in\mathcal{F}_i$, one has $\frac{1}{v}L \cap \frac{1}{v}L'=\emptyset.$\medskip

\item 
$E$ is covered by the families $\mathcal{F}_i$, i.e.
\begin{equation}\label{besi}
 E\subset  \bigcup_{1\leq i\leq Q_{d,v}}\bigcup_{L\in \mathcal{F}_i}L.
 \end{equation}
\end{itemize}
\end{proposition}
 
The case $v=1$ corresponds to  the standard Besicovich's covering lemma (see \cite{Ma}, Chapter 2, pp. 28-34 for instance).  

The proof of Proposition \ref{besimodi} relies on the following geometric lemma, which will also be used.

\begin{lemme} \label{dimconst}
For any $0<v\leq 1$ there exists a constant $\gamma_{v,d} >0 $ depending only on $v$ and the dimension $d$ only, satisfying the following: if  a  family of balls $\mathcal{B} =(B_n )_{n\in\mathbb{N}}$  and a ball $B$  are such that 

\begin{itemize}
\item
$\forall \ n\geq 1$, $\vert B_n\vert \geq \frac{1}{2}\vert B\vert,$

\item
 $\forall \ n_1 \neq n_2 \geq 1 $, $vB_{n_1}\cap vB_{n_2}=\emptyset,$ 
\end{itemize}
then $B$  intersects  at most $\gamma_{v,d}$ balls of $\mathcal{B}$.  
\end{lemme}

Note that, if one must rename the constants, it is possible to take $\gamma_{v,d}=Q_{d,v}$ for any $v\geq 1,$ which we will do.

Let us first prove the above estimates for balls.
\begin{proposition} 
\label{autosim2}
Let $\mu$ be a weakly conformal  measure as in Definition \ref{def-ssmu}.   For any $0<\varepsilon\leq \dim(\mu)$, any $0\leq \varepsilon^{\prime}\leq \frac{1}{2}$ such that $\dim (\mu)-\varepsilon+\varepsilon^{\prime}>0$, there exists a constant $\chi( d,\mu,\varepsilon,\varepsilon^{\prime})>0 $ such that for any ball $B=B(x,r)$ with $x\in K$ (the attractor of the underlying IFS) and $r\leq 1$,   one has $$\chi( d,\mu,\varepsilon,\varepsilon^{\prime})\vert B\vert ^{\dim(\mu)-\varepsilon+\varepsilon^{\prime}}\leq\mathcal{H}^{\mu, \dim(\mu)-\varepsilon }_{\infty}(\widering{B})\leq  \mathcal{H}^{ \mu, \dim(\mu)-\varepsilon}_{\infty}(B)\leq\vert B\vert ^{\dim(\mu)-\varepsilon}.$$
In addition, for any $s>\dim( \mu)$, $\mathcal{H}^{\mu,s}_{\infty}(B)=0.$
\end{proposition}


\begin{proof}

Note first that item (5) of Proposition \ref{propmuc} implies that for any $s>\dim(\mu)$,  $\mathcal{H}^{\mu,s}_{\infty}(B)=0.$

Let us consider $0\leq s<\dim_H (\mu)$ and start by few remarks.

Set $\alpha=\dim(\mu)$ and let $\varepsilon>0$ and $\rho>0$ be two real numbers. One defines

$$E_{\mu} ^{\alpha, \rho,\varepsilon}=\left\{x\in\mathbb{R}^d :\forall r\leq \rho, \ \mu\left(B\left(x,r\right)\right)\leq r^{\alpha-\varepsilon}\right\}.$$

Since $\mu$ is $\alpha$-exact dimensional, for $\mu$-almost every $x$, $\lim_{r\to 0^+}\frac{\log \mu \left(B(x,r)\right)}{\log r}=\alpha.$ This implies that, for very $\varepsilon>0,$ $\mu\left(\bigcup_{\rho>0}E_{\mu} ^{\alpha, \rho,\varepsilon}\right)=1.$

 Let $\varepsilon>0$ and $0<\rho_{\varepsilon}\leq 1$ be two real numbers such that $\mu(E_{\mu} ^{\alpha, \rho_{\varepsilon},\varepsilon})\geq \frac{1}{2}$ and write $E=E_{\mu}^{\alpha ,\rho_{\varepsilon}, \varepsilon}.$ 
 
 Write $c_{\underline{i}}=\vert f_{\underline{i}}(K)\vert.$ Let us fix $\underline{i}=(i_1,...,i_k)\in\Lambda^{*}.$ For any $x\in K$ and $r>0$, by \eqref{equreg} and \eqref{equaxibar}  applied with $\theta=\varepsilon^{\prime}$, one has
 \begin{align*}
 f_{\underline{i}}(B(x,r))\supset B(f_{\underline{i}}(x_0),\widehat{C}_{\varepsilon'}c_{\underline{i}}(x_0)^{1-\varepsilon'}r)\supset B\left(f_{\underline{i}}(x_0),\frac{\widehat{C}_{\varepsilon'}^{-\frac{2}{1-\varepsilon'}}}{\vert K\vert^{-\frac{1+\varepsilon'}{1-\varepsilon'}}} c_{\underline{i}} ^{\frac{1+\varepsilon'}{1-\varepsilon'}}r\right).
 \end{align*}

Recall that $\varepsilon'\leq \frac{1}{2}.$ Since $\frac{1+\varepsilon'}{1-\varepsilon'}\leq 1+4\varepsilon'$,
\begin{equation}
\label{includfibar}
f_{\underline{i}}(B(x,r))\supset B\left(f_{\underline{i}}(x_0),\widehat{C}_{\varepsilon'}^{-\frac{2}{1-\varepsilon'}}\cdot\vert K\vert^{\frac{1+\varepsilon'}{1-\varepsilon'}} c_{\underline{i}} ^{1+4\varepsilon'}r\right).
\end{equation}

Writing $\mu_{\underline{i}}=\mu(f_{\underline{i}}^{-1})$,  \eqref{includfibar} yields

\begin{align}
\nonumber E_{\underline{i}}&:=f_{\underline{i}}(E)\\
\nonumber &=\left\{f_{\underline{i}}(x)\in K \ : \ \forall \  r\leq \rho_{\varepsilon}, \  \mu\Big(B(x,r)\Big)\leq r^{\alpha -\varepsilon}\right\}\\
 \nonumber &\subset\left\{f_{\underline{i}}(x), x\in K \ : \ \forall \ r \leq \rho_{\varepsilon},\right.  \\ \nonumber &\left. \mu \Big(f_{\underline{i}} ^{-1}\left(B\left(f_{\underline{i}}(x_0),\widehat{C}_{\varepsilon'}^{-\frac{2}{1-\varepsilon'}}\cdot\vert K\vert^{\frac{1+\varepsilon'}{1-\varepsilon'}} c_{\underline{i}} ^{1+4\varepsilon'}r\right) \right)\Big)\leq \left(\frac{\widehat{C}_{\varepsilon'}^{-\frac{2}{1-\varepsilon'}}\cdot\vert K\vert^{\frac{1+\varepsilon'}{1-\varepsilon'}} c_{\underline{i}} ^{1+4\varepsilon'}r}{\widehat{C}_{\varepsilon'}^{-\frac{2}{1-\varepsilon'}}\cdot\vert K\vert^{\frac{1+\varepsilon'}{1-\varepsilon'}} c_{\underline{i}} ^{1+4\varepsilon'}}\right)^{\alpha -\varepsilon}\right\}\\
\label{eq77} &=\left\{y\in f_{\underline{i}}(K) \ : \ \forall \ r'\leq \widehat{C}_{\varepsilon'}^{-\frac{2}{1-\varepsilon'}}\cdot\vert K\vert^{\frac{1+\varepsilon'}{1-\varepsilon'}} c_{\underline{i}} ^{1+4\varepsilon'}\rho_{\varepsilon},\right. \\
\nonumber & \left. \  \mu_{\underline{i}}\Big(B(y,r')\Big)\leq \left(\frac{r'}{\widehat{C}_{\varepsilon'}^{-\frac{2}{1-\varepsilon'}}\cdot\vert K\vert^{\frac{1+\varepsilon'}{1-\varepsilon'}} c_{\underline{i}} ^{1+4\varepsilon'}}\right)^{\alpha -\varepsilon}\right\}.
\end{align}

Notice also that 
$$\mu_{\underline{i}}(E_{\underline{i}})=\mu(E)\geq \frac{1}{2}.$$

Let us emphasize that iterating  equation \eqref{def-ssmu2} gives 
$$\mu=\sum_{\underline{i}^{\prime}\in\Lambda ^k}p_{\underline{i}^{\prime}}\mu_{\underline{i}^{\prime}}, $$
which implies that $\mu_{\underline{i}}$ is absolutely continuous with respect to $\mu$ (since all  $p_{\underline{i}}$'s are strictly positive).  

\sk

We are now ready to estimate the $\mu$-essential content of a ball $B$ centered in $K$. 

Let us write
\begin{equation}
\label{gammasepsi}
\gamma(S,\varepsilon')=\widehat{C}_{\varepsilon'}^{-\frac{2}{1-\varepsilon'}}\cdot\vert K\vert^{\frac{1+\varepsilon'}{1-\varepsilon'}}.
\end{equation}

\sk

Let $B=B(x,r)$ with $x\in K$ and $r\leq c_0:=\min_{z\in K}\min_{1\leq i\leq m}\normi f_i '(z) \normi$.

\sk

Since $x\in K$, there exists  $\underline{i }=(i_1,...,i_k)\in\Lambda^{*}$ such that

\begin{itemize}
 \item $x\in f_{\underline{i}}(K)$,\sk 
 
 \item $\vert f_{\underline{i}}(K)\vert \leq \frac{1}{3}\vert B\vert$, \sk 
 \item $\vert f_{(i_1,...,i_{k-1})}(K)\vert\geq \frac{1}{3}\vert B\vert.$  
 \end{itemize}

By \eqref{equaxibar}, for any $y\in K$ one has
\begin{equation}
\label{math1}
\vert f_{  \underline{i}}(K)\vert \geq \widehat{C}_{\varepsilon'}^{-1}\vert \vert f_{ \underline{i}}'(y)\vert \vert^{1+\varepsilon'}\vert K\vert,
\end{equation}
and 
\begin{align}
\label{math2}
\vert \vert f_{\underline{i}}(y)\vert\vert&=\vert\vert f'_{(i_1,...,i_{n-1})}(f_n (x))\circ f'_{i_n}(x) \vert\vert\geq \vert\vert f'_{(i_1,...,i_{n-1})}(f_n (x))\vert\vert c_0 \nonumber\\
&\geq \vert f_{(i_1,...,i_{n-1}}(K) \vert^{\frac{1}{1-\varepsilon'}}\widehat{C}_{\varepsilon'}^{\frac{-1}{1-\varepsilon'}}\cdot \vert K\vert^{\frac{-1}{1-\varepsilon'}}c_0.
\end{align}

Combining \eqref{math1} and \eqref{math2}, one obtains
\begin{align}
\label{math3}
c_{\underline{i}}=\vert f_{\underline{i}}(K)\vert &\geq \widehat{C}_{\varepsilon'}^{-1-\frac{1+\varepsilon'}{1-\varepsilon'}}\vert K\vert^{\frac{-2\varepsilon'}{1-\varepsilon'}}c_0^{1+\varepsilon'}\vert f_{(i_1,...,i_{n-1})}(K)\vert^{\frac{1+\varepsilon'}{1-\varepsilon'}}\nonumber \\
& \geq \widehat{C}_{\varepsilon'}^{-1-\frac{1+\varepsilon'}{1-\varepsilon'}}\vert K\vert^{\frac{-2\varepsilon'}{1-\varepsilon'}}c_0^{1+\varepsilon'}r^{1+4\varepsilon'}.
\end{align}
%

Note that $E_{\underline{i}} \subset \widering{B}$.
\sk

Consider a set $A\subset B$ verifying $\mu(A)=\mu(B).$ One aims at giving a lower-bound of the Hausdorff content of $A$ which depends only on $B$, $d$, $\varepsilon$, $\varepsilon^{\prime}$ and the measure $\mu$. 

\sk

Consider a sequence of balls  $(L_n=B( x_n,\ell_n))_{n\geq 1}$ covering $A\cap E_{\underline{i}} $, such that  $\ell_{n}<\gamma(S,\varepsilon')\rho_{\varepsilon}c_{ \underline{i}}^{1+4\varepsilon'}$ and $ x_n\in A\cap E_{\underline{i}} $.  

Since $\mu_{\underline{i}}$ is absolutely continuous with respect to $\mu$, it holds that $\mu_{\underline{i}}(A)=1.$ 

By \eqref{eq77} applied to each ball $L_n$, $n\in\mathbb{N}$ , one has $\left(\frac{\vert L_n \vert}{\gamma(S,\varepsilon' )c_{\underline{i}}^{1+4\varepsilon^{\prime}}}\right)^{\alpha-\varepsilon}\geq \mu_{\underline{i}}(L_n)$, so that, recalling \eqref{math3},
\begin{align}
\label{eq78}
\sum_{n\in\mathbb{N}}| L_n| ^{\alpha-\varepsilon}\  &\geq \sum_{n\in\mathbb{N}}\left(\gamma(S,\varepsilon')c_{\underline{i}}^{1+4\varepsilon^{\prime}}\right) ^{\alpha -\varepsilon}\mu_{\underline{i}}(L _n )\geq  \left(\gamma(S,\varepsilon')c_{\underline{i}}^{1+4\varepsilon^{\prime}}\right) ^{\alpha -\varepsilon}\mu_{\underline{i}}\left(\bigcup_{n\in \N }L_n \right)\nonumber\\
&\geq \left(\gamma(S,\varepsilon')c_{\underline{i}}^{1+4\varepsilon^{\prime}}\right) ^{\alpha -\varepsilon}\mu_{\underline{i}}(E_{\underline{i}})\geq \frac{1}{2}\left(\gamma(S,\varepsilon')c_{\underline{i}}^{1+4\varepsilon^{\prime}}\right) ^{\alpha -\varepsilon}\nonumber
\\
&\geq \kappa(\mu,\varepsilon^{\prime},\varepsilon)r^{(1+4\varepsilon')^2(\alpha-\varepsilon)}\geq \kappa(\mu,\varepsilon^{\prime},\varepsilon)r^{(1+16\varepsilon')(\alpha-\varepsilon)} ,
\end{align}
where $\kappa(\mu,\varepsilon^{\prime},\varepsilon)=\frac{1}{2}\gamma(S,\varepsilon')^{\alpha -\varepsilon}\cdot\left(\widehat{C}_{\varepsilon'}^{-1-\frac{1+\varepsilon'}{1-\varepsilon'}}\vert K\vert^{\frac{-2\varepsilon'}{1-\varepsilon'}}c_0^{1+\varepsilon'}\right)^{(1+4\varepsilon')(\alpha-\varepsilon)}.$

This series of inequalities holds for any sequence of balls $(L_n)_{n\in\mathbb{N}}$ with radius less than $\gamma(S,\varepsilon')\rho_{\varepsilon}c_{\underline{i}}^{1+4\varepsilon^{\prime}} $  centered in $A\cap E_{\underline{i}}$. One now proves that one can freely remove those constraints on the center and the radius of the balls used to cover $A\cap E_{\underline{i}}$, up to a multiplicative constant.

Consider balls  $( L_n=B(x_n,\ell_n))_{n\geq 1}$ covering $A\cap E_{\underline{i}} $ such that  $\ell_{n}<\gamma(S,\varepsilon')\rho_{\varepsilon}c_{\underline{i}}^{1+4\varepsilon^{\prime}}  $ but  $x_n$ does not necessarily belongs to $A\cap E_{\underline{i}} $. 

Let $n\in\mathbb{N}$. One constructs recursively a sequence of balls $(L_{n,j})_{1\leq j\leq J_n}$ \ such that the following properties hold for any $1\leq j\leq J_n$:

\begin{itemize}
\item[•]  $L_{n,j}$ is centered on $A\cap E_{\underline{i}}\cap L_n$;\mk
\item[•] $A\cap E_{\underline{i}}\cap L_n \subset \bigcup_{1\leq j \leq J_n}L_{n,j}$;\mk
\item[•] for all $1\leq j\leq J_n,$ $\vert L_{n,j}\vert= \vert L_n \vert$;\mk
\item[•]the center of $L_{n,j}$ does not belong to any $L_{n,j^{\prime}}$ for $1\leq j^{\prime}\neq j\leq J_n$. \mk
\end{itemize}

To achieve this, simply consider $y_1\in A\cap E_{\underline{i}}\cap L_n$ and set $L_{1,n}=B(y_1,\ell_n).$ If $A\cap E_{\underline{i}}\cap L_n \nsubseteq L_{1,n}$, consider $y_2 \in A\cap E_{\underline{i}}\cap L_n \setminus L_{1,n}$ and set $L_{2,n}=B(y_2 ,\ell_n)$. If $A\cap E_{\underline{i}}\cap L_n \nsubseteq L_{1,n}\cup L_{2,n}$, consider $y_3 \in A\cap E_{\underline{i}}\cap L_n \setminus  L_{1,n}\cup L_{2,n}$ and set $L_{3,n}=B(y_3 ,\ell_n)$, and so on...

 Note that, for any $1\leq j\leq J_n$, any ball $L_{j,n}$ has radius $\ell_n$, intersects $L_n$ (which also has radius $\ell_n$) and, because $y_j \notin \bigcup_{1 \leq j^{\prime}\neq j\leq J_n}L_{j^{\prime},n}$, it holds that, for any $j\neq j^{\prime}$, $\frac{1}{3}L_{n,j}\cap \frac{1}{3}L_{n,j^{\prime}}=\emptyset$. By Lemma \ref{dimconst}, this implies that $J_n \leq Q_{d, \frac{1}{3}}.$

Hence, denoting by $(\widetilde L_n)_{n\in \N} $ the collection of the corresponding balls centered on $A\cap E_{\underline{i}} $ associated with all the balls $L_n$, one has by \eqref{eq78} applied to $(\widetilde L_n)_{n\in\mathbb N}$:

$$\sum_{n\in\mathbb{N}}|L_n| ^{\alpha-\varepsilon}\geq \frac{1}{Q_{d,\frac{1}{3}}} \sum_{n\in\mathbb{N}}|\widetilde L_n| ^{\alpha-\varepsilon} \geq \frac{\kappa (\mu,\varepsilon^{\prime},\varepsilon)}{Q_{d,\frac{1}{3}}} r ^{(1+4\varepsilon^{\prime})(\alpha -\varepsilon)}.$$
Remark also that any ball of radius smaller that $c_{\underline{i}}$ can be covered by at most $\left(\frac{2c_{\underline{i}}^{-4\varepsilon'}}{\gamma(S,\varepsilon')\rho_{\varepsilon}}\right)^d$ balls of radius $\gamma(S,\varepsilon')\rho_{\varepsilon}c_{\underline{i}}^{1+4\varepsilon^{\prime}}$. Moreover, by \eqref{math3},

 $$ c_{\underline{i}}^{-4\varepsilon^{\prime}}\leq \left(\widehat{C}_{\varepsilon'}^{-1-\frac{1+\varepsilon'}{1-\varepsilon'}}\vert K\vert^{\frac{-2\varepsilon'}{1-\varepsilon'}}c_0^{1+\varepsilon'}\right)^{-4\varepsilon'}r^{-4\varepsilon' \cdot(1+4\varepsilon')}.$$  

Setting
$$\widehat{\kappa}(\mu,\varepsilon,\varepsilon^{\prime},d)=\left(\frac{2\left(\widehat{C}_{\varepsilon'}^{-1-\frac{1+\varepsilon'}{1-\varepsilon'}}\vert K\vert^{\frac{-2\varepsilon'}{1-\varepsilon'}}c_0^{1+\varepsilon'}\right)^{-4\varepsilon'}}{\gamma(S,\varepsilon')\rho_{\varepsilon}}\right)^d ,$$
 
any ball of radius less than $c_{\underline{i}}$ can be covered by less than $\widehat{\kappa}(\mu,\varepsilon,\varepsilon^{\prime},d)r^{-4d\varepsilon' \cdot(1+4\varepsilon')}$ balls of radius less than $\gamma(S,\varepsilon')\rho_{\varepsilon}c_{\underline{i}}^{1+4\varepsilon^{\prime}}$.

This proves that, for any sequence of balls $\widehat{L}_n$ with $\vert\widehat{L}_n \vert\leq c_{\underline{i}}$ covering $A\cap E_{\underline{i}}$, recalling \eqref{eq78}, it holds that 

\begin{align}
\label{loceq}
\sum_{n\in\mathbb{N}}|\widehat{L}_n| ^{\alpha-\varepsilon} &\geq Q_{d,\frac{1}{3}}^{-1}\widehat{\kappa}(\mu,\varepsilon,\varepsilon^{\prime},d)^{-1}r^{4d\varepsilon' \cdot(1+4\varepsilon')}\kappa(\mu,\varepsilon^{\prime},\varepsilon)r^{(1+16\varepsilon')(\alpha-\varepsilon)}\\
&\geq  Q_{d,\frac{1}{3}}^{-1}\widehat{\kappa}(\mu,\varepsilon,\varepsilon^{\prime},d)^{-1}\kappa(\mu,\varepsilon^{\prime},\varepsilon)r^{(1+16\varepsilon')(\alpha-\varepsilon)+4d\varepsilon' \cdot(1+4\varepsilon')}.
\end{align} 
Recalling that $\vert E_{\underline{i}}\vert \leq c_{\underline{i}}$  and Definition \ref{hcont} , since \eqref{loceq} is valid for any covering $(\widehat{L}_n )_{n\in\mathbb{N}}$ of $A\cap E_{\underline{i}}$ with $\vert L_n \vert \leq c_{\underline{i}}$, one has
\begin{align}
\vert B\vert^{\alpha-\varepsilon}\geq \mathcal{H}^{\alpha-\varepsilon}_{\infty}(A)\geq \mathcal{H}^{\alpha-\varepsilon}_{\infty}(A \cap E_{\underline{i}})\geq  Q_{d,\frac{1}{3}}^{-1}\widehat{\kappa}(\mu,\varepsilon,\varepsilon^{\prime},d)^{-1}\kappa(\mu,\varepsilon^{\prime},\varepsilon)r^{(1+16\varepsilon')(\alpha-\varepsilon)+4d\varepsilon' \cdot(1+4\varepsilon')}.
\end{align}  

Taking the inf over all the set $A\subset B$ satisfying $\mu(A)=\mu(B)$, one obtains
$$\vert B\vert^{\alpha-\varepsilon}\geq \mathcal{H}^{\mu,s}_{\infty}(B)\geq  Q_{d,\frac{1}{3}}^{-1}\widehat{\kappa}(\mu,\varepsilon,\varepsilon^{\prime},d)^{-1}\kappa(\mu,\varepsilon^{\prime},\varepsilon)r^{(1+16\varepsilon')(\alpha-\varepsilon)+4d\varepsilon' \cdot(1+4\varepsilon')}. $$
The results stands for balls of diameter less than $c_0 .$ 

Set 
$$\varepsilon'_0=16\varepsilon'(\alpha-\varepsilon)+4d\varepsilon' \cdot(1+4\varepsilon')$$ and write $$\chi(d, \mu, \varepsilon,\varepsilon^{\prime}_0)= c_0 ^{\alpha-\varepsilon+\varepsilon^{\prime}_0} Q_{d,\frac{1}{3}}^{-1}\widehat{\kappa}(\mu,\varepsilon,\varepsilon^{\prime}_0,d)^{-1}\kappa(\mu,\varepsilon^{\prime}_0,\varepsilon).$$ For any ball of radius less than $1$ centered on $K$, one has 
$$\vert B\vert^{\alpha-\varepsilon}\geq \mathcal{H}^{\mu,\alpha-\varepsilon}_{\infty}(B)\geq \chi(d, \mu, \varepsilon,\varepsilon^{\prime}_0) r^{\alpha-\varepsilon+\varepsilon^{\prime}_0}. $$

\end{proof}

The estimates of Theorem \ref{contss} are now established in the case of general open sets.

Recall that by item $(5)$ of Proposition \ref{propmuc}, for any $s>\dim(\mu)$ and any set $E$, $ \mathcal{H}^{\mu,s}_{\infty}(E)=0.$ 
 \sk
 
Let us fix $s<\dim (\mu)$, $\varepsilon^{\prime}>0$ and set $\varepsilon'=\min\left\{\frac{\dim(\mu)-s}{2},\frac{1}{2} \right\}>0.$

 Since $K\cap \Omega \subset \Omega$ and $\mu(K\cap \Omega)=\mu(\Omega)$, it holds that
 $$\mathcal{H}^{\mu,s}_{\infty}(\Omega)\leq \mathcal{H}^{s}_{\infty}(\Omega \cap K).$$ 
 
 It remains to show that there exists a constant $c( d,\mu,s,\varepsilon^{\prime})$ such that for any open set $\Omega$, the converse inequality 
 
 $$c(d,\mu,s,\varepsilon^{\prime})\mathcal{H}^{s+\varepsilon^{\prime}}_{\infty}(\Omega \cap K)\leq \mathcal{H}^{\mu,s}_{\infty}(\Omega)$$
 holds.
 \sk
 
 Let  $E \subset \Omega$ be a Borel set such that $\mu(E)=\mu(\Omega)$ and  
 \begin{equation}
 \label{equat0}
 \mathcal{H}^{s}_{\infty}(E)\leq 2\mathcal{H}^{\mu,s}_{\infty}(\Omega).
 \end{equation}
  Let $\left\{L_n \right\}_{n\in\mathbb{N}}$ be a covering of $E$ by balls verifying
\begin{equation}
\label{equat1}
\mathcal{H}^{s}_{\infty}(L)\leq\sum_{n\geq 0}\vert L_n \vert^s \leq 2\mathcal{H}^{s}_{\infty}(E).
\end{equation}
 The covering $(L_n)_{n\in\mathbb{N}}$ will be modified into a covering $(\widetilde{L}_n)_{n\in\mathbb{N}}$ verifying the following properties:

\begin{itemize}
\item[•]$K\cap \Omega\subset \bigcup_{n\in\mathbb{N}}\widetilde{L}_n$,\mk
\item[•]$\bigcup_{n\in\mathbb{N}}L_n \subset \bigcup_{n\in\mathbb{N}}\widetilde{L}_n$\mk
\item[•] $$\sum_{n\geq 0}\vert \widetilde{L}_n \vert^{s+\varepsilon^{\prime}} \leq 8. 2^{s+\varepsilon^{\prime}}\frac{Q_{d,1}^2}{\chi( d,\mu,\varepsilon,\varepsilon^{\prime})}\sum_{n\geq 0}\vert L_n \vert ^s,$$
\end{itemize}

 where $Q_{d,1}$ and $\chi( d,\mu,\varepsilon,\varepsilon^{\prime})$  are the constants arising from Proposition \ref{besimodi} applied with $v=1$ and  Proposition \ref{autosim2}.
 
 Last item together with \eqref{equat0} and \eqref{equat1} then immediately imply that
  $$\frac{\chi( d,\mu,\varepsilon,\varepsilon^{\prime})}{8. 2^{s+\varepsilon^{\prime}} Q_{d,1}^2}\mathcal{H}^{s+\varepsilon^{\prime}}_{\infty}(K\cap \Omega)\leq \mathcal{H}^{\mu,s}_{\infty}(\Omega).$$

Setting $c(d,\mu,\varepsilon,\varepsilon^{\prime})=\frac{\chi( d,\mu,\varepsilon,\varepsilon^{\prime})}{8. 2^{s+\varepsilon^{\prime}} Q_{d,1}^2} $ concludes the proof.

Let us start the construction of the sequence $(\widetilde{L}_n )_{n \in\mathbb{N}}.$

Let $ \Delta=( K \setminus \bigcup_{n\in\mathbb{N}}B_n )\cap \Omega$. For every $x\in \Delta$, fix $0< r_x \leq 1$ such that $B(x,r_x)\subset \Omega.$  One of the following alternatives must occur:
\medskip
\begin{enumerate}
\item for any ball $L_n$ such that $L_n \cap B(x,r_x)\neq \emptyset$, $\vert L_n \vert \leq r_x$, or \mk 
\item there exists $n_x \in\mathbb{N}$ such that $L_{n_x}\cap B(x,r_x)\neq \emptyset$ and $\vert L_{n_x}\vert \geq r_x$.
\end{enumerate} 

Consider the set $S_1$ of points of $X$ for which the first alternative holds. 

By Lemma~\ref{besimodi} applied with $v=1$, it is possible to extract from the covering of $S_1$, $\left\{B(x,r_x),x\in S_1\right\}$, $Q_{d,1}$ families of pairwise disjoint balls, $\mathcal{F}_1 ,...,\mathcal{F}_{Q_{d,1}}$ such that
$$S_1 \subset \bigcup_{1\leq i\leq Q_{d,1}}\bigcup_{L\in\mathcal{F}_i}L.$$
 Now, any ball $L_n$ intersecting a ball $L\in \bigcup_{1\leq i\leq Q_{d,1}}\mathcal{F}_i$ must satisfy $\vert L_n \vert \leq L.$ In particular, since for any $1\leq i\leq Q_{d,1}$, the balls of $\mathcal{F}_i$ are pairwise disjoint, applying Lemma \ref{dimconst} to the ball of $\mathcal{F}_i$ intersecting $L$, we get that the ball $L_n$ intersects at most $Q_{d,1}$ balls of $\mathcal{F}_i$, hence at most $Q_{d,1}^2$ balls of  $\bigcup_{1\leq i\leq Q_{d,1}}\mathcal{F}_i .$ 
 
Let $L\in\bigcup_{1\leq i\leq Q_{d,1}}\mathcal{F}_i .$  One aims at replacing all the balls $L_n$ intersecting $L$ by the ball $2L$.

 For any $1\leq i \leq Q_{d,1}$ and any ball $L\in\mathcal{F}_i$, denote by $\mathcal{G}_L$ the set of balls $L_n$ intersecting $L$. Since $E\subset \bigcup_{n\in\mathbb{N}}L_n$ and $\mu(E)=\mu(\Omega)$, one has $E\cap L\subset \bigcup_{B\in \mathcal{G}_L}B$ and $\mu(E\cap L)=\mu(L)$.  By Definition \ref{mucont} and Proposition \ref{autosim2},  this implies that
 \begin{equation}
 \label{equat2}
\chi( d,\mu,\varepsilon,\varepsilon^{\prime})\vert L\vert^{s+\varepsilon^ {\prime}} \leq\mathcal{H}^{\mu,s}_{\infty}(L)\leq \sum_{B\in\mathcal{G}_L}\mathcal{H}^{\mu,s}_{\infty}(B)\leq \sum_{B\in\mathcal{G}_L}\vert B\vert^s.
 \end{equation}
Replace the balls of $\mathcal{G}_L$ by the ball $\widehat{L}=2L$ (recall that $\bigcup_{B\in\mathcal{G}_{L}}B \subset 2L$). The new sequence of balls so obtained by the previous construction applied to all the balls $L\in\bigcup_{\leq i\leq Q_{d,1}}\mathcal{F}_i$ is denoted by $(\widehat{L}_k)_{1\le k \le K}$, where $0\le K\le+\infty$. 

It follows from the construction and \eqref{equat2} that $S_1 \subset \bigcup_{1\le k\le K}\widehat{L}_k$ and 
\begin{equation}
\label{equat3}
\sum_{1\le k\le K}\Big (\frac{\vert \widehat{L}_k \vert}{2}\Big)^{s+\varepsilon^{\prime}} \leq \frac{Q_{d,1}^2}{\chi(d,\mu,\varepsilon,\varepsilon^{\prime})}\sum_{n\geq 0}\vert L_n \vert^s.
\end{equation}

On the other hand, since for any $x\in S_2 = \Delta\setminus S_1$, there exists $n_x\in\mathbb{N}$ such that ${L}_{n_x}\cap B(x,r_x)\neq \emptyset$ and $r_x \leq \vert {L}_{n_x}\vert$, one has $S_2\subset \bigcup_{n\in\mathbb{N}}2{L}_n$, so that 
$$
 \Big (\bigcup_{n\in\mathbb N}L_n\Big )\cup  \Big (K\cap \Omega \setminus \bigcup_{n\in\mathbb N}L_n\Big ) \subset \Big (\bigcup_{1\le k\le K}\widehat L_k\Big )\cup \Big ( \bigcup_{n\in\mathbb N}2{L}_n\Big ) .
 $$  
Putting the elements of $(\widehat{L}_k)_{1\le k \le K}$ and $(2L_n)_{n\ge 0}$ in a single sequence $(\widehat{ L}_n)_{n\ge 0}$, writing $(\widetilde{L}_n :=2\widehat{L}_n)_{n\in\mathbb{N}}$, by construction, $K\cap \Omega \subset \bigcup_{n\in\mathbb N}\widetilde{L}_n$ and due to \eqref{equat3}:

\begin{align*}
\mathcal{H}^ {s+\varepsilon^{\prime}}_{\infty}(K\cap \Omega)&\leq\sum_{n\geq 0}\vert \widetilde{L}_n \vert^{s+\varepsilon^{\prime}} \leq 2^{s+\varepsilon^{\prime}}\Big (\frac{Q_{d,1}^2}{\chi(d,\mu,\varepsilon,\varepsilon^{\prime})}+1\Big )\sum_{n\geq 0}\vert L_n \vert^s \\
&\leq 8.2^{s+\varepsilon^{\prime}} \frac{Q_{d,1}^2}{\chi( d,\mu,\varepsilon,\varepsilon^{\prime})}\mathcal{H}^{\mu,s}_{\infty}(\Omega) .
\end{align*}
The proof is concluded now by setting 
$$c( d,\mu,s,\varepsilon^{\prime})=\frac{\chi( d,\mu,\dim(\mu)-s,\varepsilon^{\prime})}{Q_{d,1}^2 4 .2^{s+\varepsilon^{\prime}}}.$$
 
 \end{proof}
 \begin{remarque}

Note that the proof for the case if open sets only relies on the fact that, there exists $\chi( d,\mu,\varepsilon,\varepsilon^{\prime})$ such that for any $x \in K$, for any $\rho>0$, there exists $0<r_x \leq \rho$ so that, writing $B=B(x,r_x),$
\begin{equation}
\label{math4}
\chi( d,\mu,\varepsilon,\varepsilon^{\prime})\vert B\vert ^{\dim(\mu)-\varepsilon+\varepsilon^{\prime}}\leq\mathcal{H}^{\mu, \dim(\mu)-\varepsilon }_{\infty}(\widering{B})\leq  \mathcal{H}^{ \mu, \dim(\mu)-\varepsilon}_{\infty}(B)\leq\vert B\vert ^{\dim(\mu)-\varepsilon}.
\end{equation}
In particular, any measure which satisfies the inequalities given in Proposition \ref{autosim2} satisfies the estimates given by Theorem \ref{contss} for any open set. Moreover, the proof of Proposition \ref{autosim2} only relies on the absolute continuity, for any $\underline{i}\in\Lambda^*$, of $\mu(f_{\underline{i}}^{-1})$ with respect to the weakly conformal measure $\mu.$ In particular Theorem  \ref{contss} actually holds for any measure $\mu \in\mathcal{M}(\mathbb{R}^d)$ for which, $\supp(\mu)\subset K$ and for any $\underline{i}\in\Lambda^*,$ $\mu(f_{\underline{i}}^{-1})$ is absolutely continuous with respect to $\mu$ (so that it holds for quasi-Bernoulli measures for instance).

 \end{remarque}
\subsection{Ubiquity results in the weakly conformal case}
Combining Theorem~\ref{lowani}  with Theorem~\ref{contss} and Lemma~\ref{equiac} yield the following result.

\begin{theoreme} 
\label{prop-ss}
 Let $S$ be a $C^{1}$ weakly conformal  IFS of a compact $X$ with attractor $K$  and $\mu$ be a self-conformal measure  associated with $S$. 

Let $(B_{n})_{n\in\mathbb{N}}$ be a sequence of closed balls centered on $K$ with $\lim_{n\to+\infty} |B_{n}| = 0$.

\begin{enumerate}
\item Suppose that $(B_n)_{n\in\N}$ is $\mu $-a.c. Then there exists a gauge function~$\zeta$ such that $\lim_{r\to 0^+}\frac{\log(\zeta(r))}{\log(r)}\ge \frac{\dim(\mu)}{\delta}$ and $\mathcal H^\zeta(\limsup_{n\to\infty}  \widering B_n ^{\delta}) >0$.  In particular 
\begin{equation}\label{weakversion}
\dim_{H}(\limsup_{n\rightarrow+\infty}\widering B_n ^{\delta})\geq \frac{\dim(\mu)}{\delta} .
\end{equation}

\item Suppose that $\mu(\limsup_{n\rightarrow+\infty}B_n)=1$. Then,  \eqref{weakversion} still holds but the existence of the gauge function is not ensured. Furthermore if $\mu$ is doubling,  then $(B_n)_{n\in\N}$ is $\mu$-a.c, so that the conclusion of  item (1) holds.
\end{enumerate}
\end{theoreme}

\begin{remarque}
One emphasizes that, for the purpose of this article, the results are stated for balls but Theorem \ref{contss} and Theorem \ref{lowani} allows to deal with more general open sets. For instance, given $1\leq \tau_1 \leq ... \leq \tau_d$, if $U_n$ is an open rectangle of length-sides $\prod_{i=1}^n \vert B_n \vert^{\tau_i},$ one needs to estimates the (classical) Hausdorff content of the union of the cubes $C\subset U_n$ of length-side $\vert B_n \vert^{\tau_d}$ (the smallest side of $U_n$) for which $C\cap K \neq \emptyset.$ This is not to hard to achieve as soon as the rectangle has sides in ``natural directions'' for the IFS we consider.
\end{remarque}

\begin{proof}

Let $\mu$ be a self-conformal measure of support $K$. 

One proves the first item of Theorem \ref{prop-ss}. 

Let $(B_n)_{n\in\mathbb{N}}$ be a $\mu$-a.c sequence of balls centered on $K$ satisfying $\vert B_n \vert \to 0$. Let us fix $\varepsilon>0$.
 
Let us start with a lemma whose proof can be found in \cite{ED2}.

\begin{lemme} 
\label{gscainf}
Let   $\mu  \in\mathcal{M}(\R^d)$.
Let    $\mathcal{B} =(B_n :=B(x_n ,r_n))_{n\in\mathbb{N}}$ be a $\mu $-a.c sequence of balls of $ \R^d$
 Then  for every $\varepsilon>0$, there exists  a  $\mu $-a.c sub-sequence $(B_{\phi(n)})_{n\in\mathbb{N}}$ of  $\mathcal{B} $ such that  for every $n\in\mathbb{N}$, $\mu (B_{\phi(n)})\leq (r_{\phi(n)})^{\underline{\dim}_H (\mu )-\varepsilon}.$
\end{lemme} 
 
By Lemma \ref{gscainf}, up to an extraction, one can assume that $\mu(B_n)\leq \vert B_n\vert^{\dim (\mu)-\frac{\varepsilon}{4}}.$ 

The following proposition is proved in \cite{ED3}.
\begin{proposition}
\label{propmuc}
 {Let $\mu \in\mathcal{M}(\R^d)$, $s\geq 0$ and $A\subset \R^d $ be a Borel set. The $s$-dimensional $\mathcal{H}^{\mu,s}_{\infty}(\cdot)$ outer measure satisfies the following properties:}
\begin{enumerate}
\item  { If $\vert A\vert\leq 1$, the mapping $s\geq 0\mapsto \mathcal{H}^{\mu,s}_{\infty}(A)$ is decreasing from $\mathcal{H}^{\mu,0}_{\infty}(A)=1$ to $\lim_{t\to +\infty}\mathcal{H}^{\mu,t}_{\infty}(A)=0$.\mk
\item    $0\leq \mathcal{H}^{\mu,s}_{\infty}(A)\leq \min\left\{\vert A\vert ^s, \mathcal{H}^{s}_{\infty}(A)\right\}$.

\mk
\item For every subset $ B\subset A$ with $\mu(A)=\mu(B)$,   $\mathcal{H}^{\mu,s}_{\infty}(A)=\mathcal{H}^{\mu,s}_{\infty}(B).$}

\mk
\item For every  $\delta \geq 1$,  $\mathcal{H}^{\mu,\frac{s}{\delta}}_{\infty}(A)\geq (\mathcal{H}^{\mu,s}_{\infty}(A))^{\frac{1}{\delta}}.$

\mk
\item  For every  $  s>\overline{\dim}_H (\mu)$,   $\mathcal{H}^{\mu,s}_{\infty}(A)=0.$

%
%

\end{enumerate}
\end{proposition}

Also, by Theorem \ref{contss} and item $(5)$ of Proposition \ref{propmuc}, there exists a constant $c(d,\mu,\dim (\mu)-\frac{\varepsilon}{2},\frac{\varepsilon}{4})$  such that, for any $n\in\mathbb{N}$, for any $\delta>1$
$$\mathcal{H}^{\mu,\frac{\dim(\mu)-\varepsilon}{\delta}}(\widering{ B_n}  ^{\delta})\geq c(d,\mu,\dim (\mu)-\varepsilon,\frac{\varepsilon}{2})\vert B_n \vert^{\dim(\mu)-\frac{\varepsilon}{2}}.$$
Taking $n$ large enough so that $\vert B_n \vert^{-\frac{\varepsilon}{4}}\geq c(d,\mu,\dim (\mu)-\frac{\varepsilon}{2},\frac{\varepsilon}{4})$, one gets

\begin{equation}
\label{jcpa}
\mathcal{H}^{\mu,\frac{\dim(\mu)-\varepsilon}{\delta}}(\widering{ B_n} ^{\delta})\geq \vert B_n \vert^{\dim(\mu)-\frac{\varepsilon}{4}}\geq \mu(B_n).
\end{equation}
Defining $\mathcal{U}_{\delta}=(\widering{ B_n}^{\delta})_{n\in\mathbb{N}}$, using \eqref{jcpa} and Theorem \ref{lowani} with $s_{ \varepsilon}= \frac{\dim (\mu)-\varepsilon}{\delta}$ and letting $\varepsilon\to 0$ finishes the proof of the first item.

Assume now that the sequence $(B_n)_{n\in\mathbb{N}}$ satisfies only $\mu(\limsup_{n\rightarrow+\infty}B_n).$ Then, by Proposition \ref{equiac} $(2B_n)_{n\in\mathbb{N}}$ is $\mu$-a.c.

Since, for any $\varepsilon>0$,
$$\limsup_{n\rightarrow+\infty}(2B_n )^{\delta+\varepsilon}\subset \limsup_{n\rightarrow+\infty}B_n^{\delta},$$
applying the first item of Theorem \ref{prop-ss} to $(2B_n )_{n\in\mathbb{N}}$, one gets 
$$\dim_H (\limsup_{n\rightarrow+\infty}B_n ^{\delta})\geq \frac{\dim (\mu)}{\delta +\varepsilon}.$$
Since $\varepsilon$ was arbitrary, the second item is proved.
\end{proof}

\begin{remarque}
The proof of Theorem \ref{prop-ss} actually shows more. With the notation of \cite[Definition 2.5]{ED3}, it is proved that $s(\mu,\mathcal{B},\mathcal{U}_{\delta})\geq \frac{\dim_H (\mu)}{\delta}$ so that \cite[Theorem 2.11]{ED3} holds for self-conformal measures instead of self-similar measure. 
\end{remarque}
\section{Proof of Theorem \ref{SHRTARG}}
\label{secprshrtag}

\sk

Write $s=\dim_H (K).$ 

The notations of the proof of Theorem \ref{contss} are adopted in this section.

\sk

Let us start by treating the easier cases $\delta<1$ and $x_0\notin K.$ 
 
 \begin{lemme}
\label{badapproxnotinK}
For any $x_0\in U$ and any $\delta<1$, $\limsup_{\underline{i}\in\Lambda^*}B(f_{\underline{i}}(x_0),\vert f_{\underline{i}}(K)\vert^{\delta})=K.$

For any $\delta>1,$ and any $x_0\notin K,$ one has 
$$\limsup_{\underline{i}\in\Lambda^*}B(f_{\underline{i}}(x_0),\vert f_{\underline{i}}(K)\vert^{\delta})=\emptyset.$$
\end{lemme}
\begin{proof}
We start by dealing with the case $\delta<1.$ Let $x_0\in U.$

Note first that  $\limsup_{\underline{i}\in\Lambda^*}B(f_{\underline{i}}(x_0),\vert f_{\underline{i}}(K)\vert^{\delta})\subset K.$ We now prove the converse inclusion.

Let $c>1$. By Lemma \ref{lemmewef} and Remark \ref{remX} applied with $X= \bigcup_{\underline{i}\in\Lambda ^*}f_{\underline{i}}(x_0)\cup K,$ there exists $D(c)$ such that for any $y\in K$ and any $\underline{i}=(i_1,...,i_n)\in\Lambda^*,$
\begin{equation}
\label{maths1}
\vert \vert f_{\underline{i}}(x_0)-f_{\underline{i}}(y)\vert\vert\leq D(c)c^n \vert\vert f'_{\underline{i}}(y)\vert\vert\cdot \vert\vert x-y\vert\vert.
\end{equation}
By Lemma \ref{lemmewef}, \eqref{bdist2}, one has
\begin{equation}
\label{maths2}
\vert\vert f_{\underline{i}}'(y)\vert\vert\leq D(c)c^n \vert f_{\underline{i}}(K)\vert.
\end{equation}
Combining \eqref{maths1} and \eqref{maths2}, one gets
\begin{equation}
\label{maths3}
\vert \vert f_{\underline{i}}(x_0)-f_{\underline{i}}(y)\vert\vert\leq  \max_{z\in K}d(x,z)D(c)^2 c^{2n}\vert f_{\underline{i}}(K)\vert.
\end{equation}
Recalling that there exists $0<t_1<t_2$ so that, uniformly on  $n$ and $\underline{i}\in \Lambda^n,$
$$t_1\leq \frac{\log \vert \vert f_{\underline{i}}\vert\vert}{n}\leq t_2,$$
taking $c=e^{\frac{t_1 \varepsilon}{4}}$ and writing $\kappa(S,\varepsilon,x_0)=\max_{z\in K}d(x,z)D(c)^2,$
\begin{equation}
\label{maths4}
\vert \vert f_{\underline{i}}(x_0)-f_{\underline{i}}(y)\vert\vert\leq \kappa(S,\varepsilon,x_0)\vert f_{\underline{i}}(K)\vert^{1-\frac{\varepsilon}{2}}.
\end{equation}
In particular, for $n$ large enough, for any $\underline{i}\in\Lambda^{n}$, $$f_{\underline{i}}(K)\subset B(f_{\underline{i}}(x_0),\vert f_{\underline{i}}(K)\vert^{1-\varepsilon}).$$
One concludes that $K\subset\limsup_{\underline{i}\in\Lambda^*}B(f_{\underline{i}}(x_0),\vert f_{\underline{i}}(K)\vert^{1-\varepsilon}).$ This ends the proof of the case $\delta<1.$

Let us treat the case $\delta>1$ and $x_0 \notin K.$

We proceed by contradiction. Assume that $\limsup_{\underline{i}\in\Lambda^*}B(f_{\underline{i}}(x_0),\vert f_{\underline{i}}(K)\vert^{\delta})\neq \emptyset.$ Consider $x\in\limsup_{\underline{i}\in\Lambda^*}B(f_{\underline{i}}(x_0),\vert f_{\underline{i}}(K)\vert ^{\delta}).$ By compacity of $\Lambda^ {\mathbb{N}}$, there exists   $\underline{i}=(i_1,...)\in \Lambda^{\mathbb{N}}$ such that for an infinity of integer $k\in\mathbb{N},$ $$x\in B(f_{i_1}\circ ...\circ f_{i_k}(x_0),\vert f_{\underline{i}}(K)\vert^{\delta}).$$ Note that this implies that $x =\lim_{k\rightarrow+\infty}f_{i_1}\circ ...\circ f_{i_k}(x_0),$ so that for any $n \in\mathbb{N},$ $$x=f_{i_1}\circ ...\circ f_{i_n}(\lim_{k\rightarrow+\infty}f_{i_{n+1}}\circ ...\circ f_{i_{n+k}}(x_0)).$$ Writing $z_n=\lim_{k\rightarrow+\infty}f_{i_{n+1}}\circ ...\circ f_{i_{n+k}}(x_0)\in K,$ one has $$x=f_{i_1}\circ ...\circ f_{i_n}(z_n). $$
 It follows that, for any $n\in\mathbb{N},$ 
\begin{equation}
\label{distxzn} 
 d(f_{i_1}\circ ...\circ f_{i_n}(x_0),x)=d(f_{i_1}\circ ...\circ f_{i_n}(x_0),f_{i_1}\circ ...\circ f_{i_n}(z_n)).
 \end{equation} 
Write $\underline{i}=(i_1,...,i_n)$ and let $ \varepsilon>0$ be small enough so that $1\leq \frac{1+\varepsilon}{1-\varepsilon}<\delta$. By \eqref{equreg} and \eqref{equaxibar} applied with $\theta=\varepsilon$, 
 \begin{align*}
 d(f_{\underline{i}}(x_0),f_{\underline{i}}(z_n))&\geq \widehat{C}_{\varepsilon}^{-1}\vert\vert f_{\underline{i}}'(x)\vert\vert^{1+\varepsilon}\vert\vert x-z_n\vert\vert \\
 &\geq d(x,K)\widehat{C}_{\varepsilon}^{-1}\left(\vert K\vert \widehat{C}_{\varepsilon}\right)^{\frac{-1-\varepsilon}{1-\varepsilon}}\vert f_{\underline{i}}(K)\vert^{\frac{1+\varepsilon}{1-\varepsilon}}
 \end{align*}
 which implies that, for $n$ large enough, 
 
\begin{equation}
d(f_{\underline{i}}(x_0),x)>\vert f_{\underline{i}}(K)\vert^{\delta}.
\end{equation}
This is a contradiction. One concludes that
$$\limsup_{\underline{i}\in\Lambda^*}B(f_{\underline{i}}(x_0),\vert f_{\underline{i}}(K)\vert^{\delta})=\emptyset.$$
\end{proof}

\begin{remark}
\label{preccasss}
In the case where $S=\left\{f_1,...,f_m\right\}$ is a self-similar system, more precise statement can be made. Denote $0<c_1,...,c_m <1$ the contracting ratio of $f_1,...,f_m.$ In the self-similar case, for any $z\in K$ and any $\underline{i}\in\Lambda^{*}$
\begin{align*}
d(f_{\underline{i}}(x_0),f_{\underline{i}}(z))=c_{\underline{i}}d(x,z)\leq c_{\underline{i}}\max_{y\in K}d(y,x).
\end{align*}
This implies that, writing $C(x_0,S)=\max_{y\in K}d(y,x)$, $f_{\underline{i}}(K)\subset B(f_{\underline{i}}(x_0),C(x_0,S)c_{\underline{i}})$ and
$$K=\limsup_{\underline{i}\in\Lambda^*}B(f_{\underline{i}}(x_0),C(x_0,S)c_{\underline{i}}). $$
\end{remark}

The following subsections are dedicated to the proof of the case $x_0\in K$ and $\delta\geq 1$ of Theorem \ref{SHRTARG}.

\subsection{Variational principle and $C^1$ weakly conformal IFS }
A modified version of a proposition of Feng and Hu, used in their proof their variational principal \cite[Theorem 2.13]{FH} is needed. The following subsection is dedicated to this  modification.

The following proposition will be slightly modified so that the measure involved is fully supported on $K$.
\begin{proposition}[\cite{FH}]
\label{varprinc}
For any $\varepsilon>0$, there exists $n_{\varepsilon}\in\mathbb{N}$ as well as words $\underline{i}_{1},...,\underline{i}_{n_{\varepsilon}}\in\Lambda^{*}$ such that:

\begin{itemize}
\item[•] for any $1\leq j<j^{\prime}\leq n_{\varepsilon}$, $f_{\underline{i}_j}(K)\cap f_{\underline{i}_{j^{\prime}}}(K)=\emptyset,$\sk
\item[•] writing $S_{\varepsilon}=\left\{f_{\underline{i}_1},...,f_{\underline{i}_{n_{\varepsilon}}}\right\}$, there exists a probability vector $P_{\varepsilon}=(p_1 ,...,p_{n_{\varepsilon}})$ such that the weakly conformal measure $\mu_{\varepsilon}$ associated with $P_{\varepsilon}$ and $S_{ \varepsilon}$ satisfies $\dim_H (\mu_{ \varepsilon})\geq \dim_H (K)-\varepsilon.$
\end{itemize}

\end{proposition}

Let us remark that, due to the the first item, the IFS $S_{\varepsilon}=\left\{T_ 1,...,T_{n_{\varepsilon}}\right\}$ satisfies the strong separation condition, hence the dimension of a weakly conformal measure depends continuously on the choice of the probability vector. Moreover, writing $\nu_{\varepsilon}>0$ the canonical measure on the coding associated with $\mu_{\varepsilon}$, then there exists $\lambda_{\nu_{\varepsilon}}$ (see \cite{FH}), such that for $\nu_{\varepsilon}$-almost $(x_n)_{n\in\mathbb{N}}$, it holds that 
$$\lim_{n\rightarrow+\infty}\frac{\log \vert T_{x_1}\circ ...\circ T_{x_n}(K) \vert}{n}=-\lambda_{\nu_{\varepsilon}}.$$

\sk

One first starts by proving the following modified version of Theorem \ref{varprinc}.
\begin{proposition}
\label{modivar}
Let $\varepsilon_0>0$. There exists an IFS $S_{\varepsilon_0}$ and a weakly conformal  measure $\mu_{\varepsilon_0}$ (associated with  $S_{\varepsilon_0}$) such that $\supp(\mu_{\varepsilon_0})=K$ and $\dim_H (\mu_{\varepsilon_0})\geq s-\varepsilon_0.$
\end{proposition}

\begin{remarque}
Similarly to  the proof of \cite[Theorem 2.13]{FH}, Proposition \ref{modivar}  provides a measure on $\Lambda^{\mathbb{N}}$ and taking weak limits of ergodic averages of this measure gives an ergodic measure fully supported on $K$ with dimension larger $s-\varepsilon.$
\end{remarque}

\begin{proof}
Fix $\varepsilon=\frac{\varepsilon_0}{2}>0$. Consider $S_{\varepsilon}=\left\{f_{\underline{i}_1},...,f_{\underline{i}_{n_{\varepsilon}}}\right\}$, $P_{\varepsilon}$, $\mu_{\varepsilon}$ as in Theorem \ref{varprinc} and $0<\varepsilon'<\frac{n_{\varepsilon}}{5m}\cdot \min_{1\leq i\leq m}p_i$.

Set $$\begin{cases}g_{j}=f_j \text{ for }1\leq i\leq m \\ g_{j}=f_{\underline{i}_j}\text{ for }m+1\leq i\leq n_{\varepsilon}\end{cases},$$

 $\widetilde{S}_{\varepsilon}=\left\{g_1 ,...,g_{m+n_{\varepsilon} } \right\}$ and denote by  $\widetilde{P}_{\varepsilon,\varepsilon^{\prime}}=(\widetilde{p}_1 ,...,\widetilde{p}_{m+n_{\varepsilon}})$ the probability vector defined as
$$\begin{cases}\widetilde{p}_j =\varepsilon^{\prime}\text{ for }1\leq j \leq m \\ \widetilde{p}_j =p_{j-m}-\frac{m}{n_{\varepsilon}}\varepsilon^{\prime}.
\end{cases}$$

Let $\mu_{\varepsilon,\varepsilon^{\prime}}$ be the weakly conformal measure associated with $\widetilde{S}_{\varepsilon}$ and $\widetilde{P}_{\varepsilon,\varepsilon^{\prime}}.$ Applying Theorem Theorem \ref{decompofeng} to $\mu_{\varepsilon,\varepsilon^{\prime}}$, let us prove that the corresponding $h$ (see second item of Theorem \ref{decompofeng}) goes to 0 as $\varepsilon^{\prime}$ goes to $0$.

\sk

Set $\Theta=\left\{1,...,n_{\varepsilon}+m\right\}$ and $\Theta^{*}=\bigcup_{k>0}\Theta^k .$ Let  us denote $\pi_{\Theta}$ the canonical projection. One endows $\Sigma_{\Theta}=\Theta^{\mathbb{N}}$ with metric $d_{\Theta}$ defined by, for any $x=(x_n),y=(y_n)\in\Sigma_{\Theta},$ $d_{\Theta}(x,y)=e^{-\min\left\{i\in\mathbb{N}:x_i \neq y_i\right\}}$ and $d_{\Theta}(x,x)=0$.

Let $\nu_{\varepsilon,\varepsilon^{\prime}}\in \mathcal{M}(\Theta^{\mathbb{N}})$ be the Bernoulli product verifying $\nu_{\varepsilon,\varepsilon^{\prime}}\circ \pi_{\Theta}^{-1}=\mu_{\varepsilon,\varepsilon^{\prime}}.$


 By the strong law of large number, for  every $x=(x_n)_{n\in\mathbb{N}}$ in a set $\widetilde{\Sigma}_{\Theta}$ of $\nu^{\varepsilon,\varepsilon^{\prime}}$-full measure, there exists $N_x \in\mathbb{N}$ such that for any $n\geq N_x,$ any $1\leq i \leq n_{\varepsilon}+m$, 
 
 \begin{equation}
 \label{freq}
 \left\vert\frac{\#\left\{1\leq j\leq n \ : x_j =i\right\}}{n}-\widetilde{p}_i\right\vert\leq\varepsilon^{\prime}.
 \end{equation}
For $n\in\mathbb{N}$, write
$$A_n=\left\{x\in\widetilde{\Sigma}_{\Theta}:N_x \leq n\right\} .$$ 
 
By Theorem \ref{decompofeng}, there exists $N$ such that, using the notation involved,  
 $$\mu_{\varepsilon,\varepsilon^{\prime}}\left(B_N=\left\{y:   \dim_H(\mu_{\varepsilon,\varepsilon'}^{\pi_{\Theta}^{-1}(\left\{y\right\})})=h \text{ and }\mu_{\varepsilon,\varepsilon'}^{\pi_{\Theta}^{-1}(\left\{y\right\})} (A_N)\geq \frac{1}{2}\right\}\right)\geq \frac{1}{2}.$$
 We fix such an $N$.

We need the following lemma.
\begin{lemme}
\label{lemmjcpa}
Consider $y\in B_N$, $x=(x_n)_{n\in\mathbb{N}}\in \pi_{\Theta}^{-1}(\left\{y\right\})\cap A_N $ and $N^{\prime}\geq N.$ 

Let $I_{N^{\prime}}((x_n)_{n\in\mathbb{N}})=\left\{1\leq k\leq N^{\prime}: 1 \leq x_k \leq m \right\}$. Then, for any $\widetilde{x}=( \widetilde{x}_n)_{n\in\mathbb{N}}\in \pi^{-1}(y)$ and any $1\leq j\leq N^{\prime}$ such that $j \notin I_{N^{\prime}}((x_n)_{n\in\mathbb{N}}),$ one has 
$$x_j =\widetilde{x}_j .$$ 
\end{lemme}

 \begin{proof}
We proceed by contradiction. Suppose that the claim is not true and let $j_0 \geq 1$ such for any $1\leq i < j_0$, $\widetilde{x}_i = x_i$ and $x_{j_0}\neq x_{j_0}.$ Write $z=\lim_{k \to +\infty}g_{x_{j_0 +1}}\circ g_{x_{j_0 +2}}\circ ... \circ g_{x_{j_ 0 +k}}(0)$ and  $\widetilde{ z}=\lim_{k \to +\infty}g_{\widetilde{x}_{j_0 +1}}\circ g_{\widetilde{x}_{j_0 +2 }}\circ ... g_{\widetilde{x}_{j_ 0 +k}}(0).$ Then, recalling that $x,\widetilde{x}\in \pi_{\Theta}^{-1}(\left\{y\right\}),$ $$g_{x_1}\circ ... \circ g_{x_{j_0-1}}\circ g_{x_{j_0}}(z)=g_{x_0}\circ ...\circ g_{x_{j_0-1}} \circ g_{\widetilde{x}_{j_0}}(\widetilde{z})=y,$$ which implies that $g_{x_{j_0}}(z)=g_{\widetilde{x}_{j_0}}(\widetilde{z})$, yielding to a contradiction since $g_{x_{j_0}}(K)\cap g_{ \widetilde{x}_{j_0}}(K)=\emptyset.$ 
 \end{proof}

Continuing the proof of the proposition, we note also that, by \eqref{freq}, for every $x\in\widetilde{\Sigma}_{\Theta}$ and $N'\geq N,$
\begin{equation} 
\label{jcpa1} 
 \#\left\{1\leq k\leq N^{\prime} \ : \ x_{k}\in\left\{1,...,m\right\}\right\}\leq 2m\varepsilon^{\prime}N^{\prime}.
 \end{equation}
Lemma \ref{lemmjcpa} together with \eqref{jcpa1} yields
\begin{align*}
\#\left\{\underline{i}\in\Theta^{N^{\prime}} \ : [\underline{i}]\cap A_{N}\cap \pi_{\Theta}^{-1}(\left\{y\right\})\right\}&\leq  \sum_{k=0}^{\lfloor 2m\varepsilon^{\prime}N^{\prime}\rfloor+1}\binom{ N^{\prime}} {k}m^{k} \\
& \leq (\lfloor 2m\varepsilon^{\prime}N^{\prime}\rfloor+2)\binom{ N^{\prime}} {\lfloor 2m\varepsilon^{\prime}N^{\prime}\rfloor+1}m^{\lfloor 2m\varepsilon^{\prime}N^{\prime}\rfloor+1},
\end{align*}
where we used that $\varepsilon'<\frac{1}{5m}$ so that $2m\varepsilon' N' <\frac{N'}{2}$, provided that $N$ was chosen large enough and, for any $0\leq k\leq \lfloor 2m\varepsilon^{\prime}N^{\prime}\rfloor+1,$ $\binom{N'}{k}\leq \binom{ N^{\prime}} {\lfloor 2m\varepsilon^{\prime}N^{\prime}\rfloor+1}.$ Using Sterling formula, there exists a constant $C>0$ such that
\begin{align}
\label{majopack}
&\#\left\{\underline{i}\in\Theta^{N^{\prime}} \ : [\underline{i}]\cap A_{N}\cap \pi_{\Theta}^{-1}(\left\{y\right\})\right\} \nonumber\\
&\leq C(\lfloor 2m\varepsilon^{\prime}N^{\prime}\rfloor+2)\frac{(N^{\prime})^{\lfloor 2m\varepsilon^{\prime}N^{\prime}\rfloor+1}\cdot m^{\lfloor 2m\varepsilon^{\prime}N^{\prime}\rfloor+1}}{\left(\frac{\lfloor 2m\varepsilon^{\prime}N^{\prime}\rfloor+1}{e}\right)^{\lfloor 2m\varepsilon^{\prime}N^{\prime}\rfloor+1}\sqrt{2\pi(\lfloor 2m\varepsilon^{\prime}N^{\prime}\rfloor+1)}}\nonumber  \\
&\leq C(\lfloor 2m\varepsilon^{\prime}N^{\prime}\rfloor+2)\left(\frac{m N^{\prime}}{\frac{2m\varepsilon^{\prime}N^{^{\prime}}}{e}}\right)^{\lfloor 2m\varepsilon^{\prime}N^{\prime}\rfloor+1}\frac{1}{\sqrt{2\pi(\lfloor 2m\varepsilon^{\prime}N^{\prime}\rfloor+1)}}\nonumber \\
& \leq C(\lfloor 2m\varepsilon^{\prime}N^{\prime}\rfloor+2)\left(\frac{e}{2\varepsilon^{\prime}}\right)^{3mN^{\prime}\varepsilon^{\prime}}\nonumber\\
&= C(\lfloor 2m\varepsilon^{\prime}N^{\prime}\rfloor+2)e^{3mN^{\prime}\varepsilon^{\prime}\log \frac{e}{2\varepsilon^{\prime}}}\leq e^{\sqrt{\varepsilon^{\prime}}N^{\prime}},
\end{align}
provided that $\varepsilon^{\prime}$ was chosen small enough at start and $N$ (so $N^{\prime}$ too) large enough.

Since \eqref{majopack} holds for any $N^{\prime}\geq N$, one obtains that 
$$\dim_P (A_{N}\cap \pi_{\Theta}^{-1}(\left\{y\right\}))\leq \sqrt{\varepsilon^{\prime}}.$$
Recalling the third item of Theorem \ref{decompofeng}, one gets 


$$h\leq \sqrt{\varepsilon^{\prime}}.$$

By Remark \ref{remboundlyapu} and the fourth item of Theorem \ref{decompofeng}, there exists a constant $\widetilde{C}$, depending on the system $S$ such that $\dim_H (\mu_{\varepsilon,\varepsilon^{\prime}})\geq \frac{\dim_H (\nu_{\varepsilon,\varepsilon^{\prime}})}{\lambda_{\nu_{\varepsilon,\varepsilon^{\prime}}}}-\widetilde{C}\sqrt{\varepsilon^{\prime}},$ where  $\lambda_{\nu_{\varepsilon,\varepsilon^{\prime}}}$ is defined by Definition \ref{expolyapu}.
Also, by Corollary \ref{convproba}, for any Bernoulli product $\nu \in\mathcal{M}(\Theta)$ associated with a probability vector $\widehat{P}\in(0,1)^{n_{\varepsilon}+m}$ the Lyapunov exponent depends continuously on the vector $\widehat{P}.$ 

Since it is also the case for $\dim_H (\nu)$ and $\lim_{\varepsilon^{\prime}\to 0}P_{\varepsilon,\varepsilon^{\prime}}=\left\{0\right\}^m \times P_{\varepsilon}$, for $\varepsilon^{\prime}$ small enough, we conclude that
$$\dim_H (\mu_{\varepsilon,\varepsilon^{\prime}})\geq \frac{\dim_H (\nu_{\varepsilon})}{\lambda_{\nu_{\varepsilon}}}-2\varepsilon \geq s-2\varepsilon,$$
which concludes the proof.
\end{proof}
\subsection{Proof of Theorem \ref{SHRTARG}}

Let us recall that, by Proposition \ref{propopres} and Definition \ref{expdim}, $\dim(S)$ verifies, for any $z\in K$,  
$$P(\dim(S))=\lim_{k\rightarrow+\infty}\frac{1}{k}\log\sum_{\underline{i}\in \Lambda^k} \vert f_{\underline{i}}(K)\vert^{\dim(S)} =0.$$

Fix $x_0\in K$ $\delta\geq 1$ and write $$\mathcal{L}(\delta)=\limsup_{\underline{i}\in \Lambda^*}B(f_{\underline{i}}(x_0),\vert f_{\underline{i}}(K)\vert^{\delta}).$$

Let us first show that $\dim_H (\mathcal{L}(\delta))\leq \frac{\dim (S)}{\delta}.$

Let $\alpha$ and $C_{\alpha}$ be as in \eqref{encafibar}, $0<t\leq \frac{\alpha}{2}$. For $k\in\mathbb{N}$, set
$$\Lambda^{(k)}_t =\left\{\underline{i}=(i_1,...,i_{\ell}): \vert f_{i_1,...,i_{\ell}}(K)\vert<  t^{k}\leq \vert f_{i_1,...,i_{\ell}-1}(K)\vert \right\}.$$
Note that, there exists $k_0 \in\mathbb{N}$ so that, for any $k\geq k_0,$ if $\underline{i}$ belongs to $\Lambda^{(k)}_t$, then, for any $1\leq j\leq m$, $\underline{i}j \notin \Lambda^{(k)}_{t}.$ In particular, for any $\underline{i}\neq \underline{j}\in \Lambda^{(k)}_t$, $[\underline{i}]\cap [\underline{j}]=\emptyset.$ This implies that, for any $\nu \in\mathcal{M}(\Lambda^{\mathbb{N}}),$
\begin{equation}
\sum_{\underline{i}\in \Lambda^{(k)}_t}\nu([\underline{i}])\leq 1.
\end{equation}

Consider $\varepsilon>0$. Let us recall that, by proposition \ref{weakconfapeupres} applied with $\varepsilon'=-\frac{\varepsilon}{2}\log t $ and $s=\dim(S)$ combined with Remark \ref{frakegibbs}, there exists $\gamma_{\varepsilon'}$ and a measure $\nu_{\varepsilon'}\in\mathcal{M}(\Lambda^{\mathbb{N}})$ such that for any $\underline{i}=(i_1,...,i_k)\in\Lambda^{*},$
\begin{equation}
\gamma_{\varepsilon'}^{-1}e^{k\frac{\varepsilon}{2}\log t}\vert f_{\underline{i}}(K)\vert^{\dim(S)}\leq \nu_{\varepsilon'}([\underline{i}])\leq \gamma_{\epsilon'}e^{-k\frac{\varepsilon}{2}\log t}\vert f_{\underline{i}}(K)\vert^{\dim(S)}.
\end{equation}


For any  $\delta\geq 1$,


\begin{align}
\label{sumkiconv}
\sum_{\underline{i}\in\in\bigcup_{k\geq k_0}\Lambda^k}\left(\vert f_{\underline{i}}(K)\vert^\delta\right)^{\frac{\dim(S)+\varepsilon}{\delta}}&=\sum_{\underline{i}\in\bigcup_{k\geq k_0}\Lambda^{(k)}_t}\vert f_{\underline{i}}(K)\vert^{\dim(S)+\varepsilon}\nonumber\\
&\leq \sum_{k\geq k_0}\sum_{\underline{i}\in\Lambda^{(k)}_t}t^{k\varepsilon}\gamma_{\varepsilon'}e^{-k\frac{\varepsilon}{2}\log t}\nu_{ \varepsilon'}([\underline{i}])\leq \gamma_{\varepsilon'}\sum_{k\geq k_0}t^{k\frac{\varepsilon}{2}}<+\infty.
\end{align}
As a consequence,
 $$\dim_H (\limsup B(f_{\underline{i}}(x_0),\vert f_{\underline{i}}(K)\vert^{\delta}))\leq \frac{\dim(S)+\varepsilon}{\delta},$$
and letting $\varepsilon$ tend to $0$ establishes the upper-bound.

\sk

Now we established the desired lower-bound for $\dim_H (\mathcal{L}(\delta)).$

  \sk

Let $\varepsilon>0$ and $\mu_{\varepsilon}$ be a weakly conformal measure as in Proposition \ref{modivar}. For any $k\in\mathbb{N}$, the balls $\left\{B(f_{\underline{i}}(x_0),\vert f_{\underline{i}}(K)\vert)\right\}_{\underline{i}\in\Lambda^{k}}$ are centered on $K=\supp (\mu)$ and cover $K$. This implies that $ \mu_{\varepsilon}(\limsup_{\underline{i}\in \Lambda^{*}}B(f_{\underline{i}}(x_0),\vert f_{\underline{i}}(K)\vert))=1.$

 Applying Theorem \ref{prop-ss}, one gets
$$\frac{s-\varepsilon}{\delta}\leq  \dim_H \left(\limsup_{\underline{i}\in \Lambda^{*}}B\big(f_{\underline{i}}(x_0),\vert f_{\underline{i}}(K)\vert^{\delta}\big)\right).$$
Letting $\varepsilon \to 0$ finishes the proof.
\section{Proof of Theorem \ref{compbaker}}
\label{seccompbaker}

\begin{proof}
Let us first check that $$\dim_H \limsup_{\underline{i}\in \Lambda^{*}}B(f_{\underline{i}}(x),(\vert f_{\underline{i}}(K))\vert g(\vert \underline{i}\vert))^{\delta\frac{s_g}{\dim(S)}})\leq \frac{\dim(S)}{\delta}:$$

Let $\varepsilon>0$. Recalling \eqref{sg}, one has  

\begin{align*}
\sum_{k\geq 0}\sum_{\underline{i}\in\Lambda^k}\Big((\vert f_{\underline{i}}(K))\vert g(k))^{\delta\frac{s_g}{\dim(S)}}\Big)^{\frac{\dim(S)+\varepsilon}{\delta}}&\leq\sum_{k\geq 0}\sum_{\underline{i}\in\Lambda^k}k\Big(\vert f_{\underline{i}}(K))\vert g(k)\Big)^{s_g\frac{\dim(S)+\varepsilon}{\dim(S)}}<+ \infty  .
\end{align*}
This proves that $\dim_H \limsup_{\underline{i}\in \Lambda^{*}}B\Big(f_{\underline{i}}(x),\big(\vert f_{\underline{i}}(K))\vert g(\vert \underline{i}\vert\big)\Big)^{\delta\frac{s_g}{\dim(S)}})\leq \frac{\dim(S) +\varepsilon}{\delta}$. Letting $\varepsilon\to 0$ concludes this part of the proof.

\sk

Now we prove that $$\dim_H \limsup_{\underline{i}\in \Lambda^{*}}B\Big(f_{\underline{i}}(x),(\vert f_{\underline{i}}(K))\vert g(\vert \underline{i}\vert))^{\delta\frac{s_g}{\dim(S)}}\Big)\geq \frac{\dim(S)}{\delta}.$$

\sk

Let $\varepsilon>0.$ Note that one has, by \eqref{sg},

$$\sum_{k\geq 0}\sum_{\underline{i}\in\Lambda^k}k\big(\vert f_{\underline{i}}(K))\vert g(k)\big)^{ \frac{s_g}{1+\varepsilon}}=+\infty.$$
By Theorem \ref{resbaker}, for any $\varepsilon>0$, 

$$\mu\Big( \limsup_{\underline{i}\in \Lambda^{*}}B(f_{\underline{i}}(x),(\vert f_{\underline{i}}(K))\vert g(\vert \underline{i}\vert))^{\frac{ s_g}{(1+\varepsilon)  \dim(S)}})\Big)=1.$$ 

Using Theorem \ref{prop-ss}, one gets 
$$\dim_H  \limsup_{\underline{i}\in \Lambda^{*}}B\Big(f_{\underline{i}}(x),\left(\vert f_{\underline{i}}(K))\vert g(\vert \underline{i}\vert)\right)^{\delta\frac{s_g}{\dim(S)}}\Big) \geq \frac{\dim_H (\mu)}{(1+\varepsilon)\delta}.$$
Letting $\varepsilon\to 0,$ one has $\dim_H  \limsup_{\underline{i}\in \Lambda^{*}}B\Big(f_{\underline{i}}(x),\left(\vert f_{\underline{i}}(K))\vert g(\vert \underline{i}\vert)\right)^{\delta\frac{s_g}{\dim(S)}}\Big)\geq \frac{\dim(S)}{\delta}$ and this ends the proof.
\end{proof}
\bibliographystyle{plain}
\bibliography{bibliogenubi}
\end{document}